\documentclass[12pt]{amsart}
\usepackage{amsmath, amsthm, amsfonts, amssymb, color}
\usepackage{mathrsfs}
\usepackage{amsfonts, amsmath}
\usepackage{amsmath,amstext,amsthm,amssymb,amsxtra}
\usepackage{txfonts} 
\usepackage[colorlinks, citecolor=blue,hypertexnames=false]{hyperref}
\usepackage{pgf,tikz}

\theoremstyle{plain}
\newtheorem{theorem}{Theorem}[section]
\newtheorem{prop}[theorem]{Proposition}
\newtheorem{lemma}[theorem]{Lemma}
\newtheorem{corollary}[theorem]{Corollary}

\newtheorem{mainresult}[theorem]{Main Result}

\theoremstyle{definition}
\newtheorem{definition}[theorem]{Definition}
\newtheorem{remark}[theorem]{Remark}

\allowdisplaybreaks

\def\RR{\mathbb R}
\def\CC{\mathbb C}

\def\la{\lambda}
\def\11{1\!\!1}

\def\SLi{\sqrt{L_1}}
\def\SLii{\sqrt{L_2}}
\def\l{\lambda}
\newcommand\D{\mathcal{D}}

\DeclareMathOperator{\support}{supp}

\title[Marcinkiewicz-type spectral multipliers]
{Marcinkiewicz-type spectral multipliers on \\
Hardy and Lebesgue spaces on product spaces of homogeneous
type}

\author{Peng Chen, Xuan Thinh Duong, Ji Li, Lesley A.~Ward and Lixin Yan }

\address{
    Peng Chen,
    Department of Mathematics,
    Sun Yat-sen University,
    Guangzhou, 510275,
    P.R.~China}
\email{
    achenpeng1981@163.com}
\address{
    Xuan Thinh Duong,
    Department of Mathematics,
    Macquarie University, NSW 2109,
    Australia}
\email{
    xuan.duong@mq.edu.au}
\address{
    Ji Li,
    Department of Mathematics,
    Macquarie University, NSW 2109,
    Australia
    }
\email{
    ji.li@mq.edu.au}
\address{
    Lesley A.~Ward,
    School of Information Technology and Mathematical Sciences,
    University of South Australia,
    Mawson Lakes SA 5095,
    Australia}
\email{
    lesley.ward@unisa.edu.au}
\address{
    Lixin Yan,
    Department of Mathematics,
    Sun Yat-sen University,
    Guangzhou, 510275,
    P.R.~China}
\email{
    mcsylx@mail.sysu.edu.cn}

\subjclass[2010]{Primary: 42B15; Secondary: 42B30, 47B15, 35P05.}

\keywords{Marcinkiewicz-type spectral multipliers, Hardy
    spaces, nonnegative self-adjoint operators, heat semigroup,
    Stein--Tomas type estimates, restriction type estimates,
    finite propagation speed property, spaces of homogeneous
    type.}

\date{\today}

\begin{document}

\begin{abstract}
    Let $X_1$ and $X_2$ be metric spaces equipped with doubling
    measures and let $L_1$ and $L_2$ be nonnegative
    self-adjoint second-order operators acting on $L^2(X_1)$
    and $L^2(X_2)$ respectively. We study multivariable
    spectral multipliers $F(L_1, L_2)$ acting on the Cartesian
    product of $X_1$ and $X_2$.
    Under the assumptions of the finite propagation speed
    property and Plancherel or Stein--Tomas restriction type
    estimates on the operators $L_1$ and~$L_2$, we show that if
    a function~$F$ satisfies a Marcinkiewicz-type differential
    condition then the spectral multiplier operator $F(L_1, L_2)$ is
    bounded from appropriate Hardy spaces to Lebesgue spaces on
    the product space $X_1\times X_2$. We apply our results to
    the analysis of second-order elliptic operators in the
    product setting, specifically Riesz-transform-like
    operators and double Bochner--Riesz means.
\end{abstract}

\maketitle

\tableofcontents

\section{Introduction}\label{sec:intro}
\setcounter{equation}{0}

Our goal in this paper is to build a theory of
Marcinkiewicz-type spectral multipliers in the general setting
of spaces of homogeneous type and self-adjoint operators. Our
key assumptions on the operators are the finite propagation
speed property and restriction type estimates. We work in the
multivariable setting and we focus on the boundedness of
spectral multipliers acting on Lebesgue spaces and on Hardy
spaces associated to the self-adjoint operators.  In this
section we review these ideas, present our results, and put
them in context.


%

Let us first explain the definition of multivariable spectral
multipliers. Suppose that $X_1\times X_2$ is the Cartesian
product of measure spaces $X_1$ and $X_2$ and that $L_1$ and
$L_2$ are two nonnegative self-adjoint operators acting on the
spaces $L^2(X_1)$ and $L^2(X_2)$, respectively. There is a
unique spectral decomposition~$E$ such that for all Borel
subsets $A\subset {\mathbb R}^2$, $E(A)$ is a projection on
$L^2({X_1\times X_2})$ and such that  for all Borel subsets
$A_i\subset {\mathbb R}$, $i = 1$, 2, one has $ E(A_1\times
A_2) = E_{L_1}(A_1)\otimes E_{L_2}(A_2). $ Hence for a bounded
Borel function $F: {\mathbb R}^2\rightarrow {\mathbb C}$, one
may define the spectral multiplier operator $F(L_1, L_2)$
acting on the space $L^2({X_1\times X_2})$ by the formula
\begin{eqnarray}\label{e1.1}
    F(L_1, L_2)
    := \int_{{\mathbb R}^2} F(\lambda_1, \lambda_2) \,
        dE(\lambda_1, \lambda_2).
\end{eqnarray}
Clearly $F(L_1, L_2)$ is bounded on $L^2({X_1\times X_2})$.

A natural question that arises is to find sufficient conditions
on the multiplier function~$F$ guaranteeing that the operator
$F(L_1, L_2)$ is bounded on $L^p(X_1\times X_2)$, for $1 < p <
\infty$. A similar question can also be asked for $F(L_1, L_2)$
acting on appropriate Hardy spaces. The desired conditions are
variants of those to be found in the multiplier theorems of
Marcinkiewicz, Mikhlin, and H\"ormander; see for
instance~\cite{COSY, Ch2, DOS, FS, GH, GS, Ho, KU, Mar, MRS1,MRS2, Ou, St1,
St2}.

The following brief summary outlines three areas of progress in
this direction.

{\textbf {i)}} One early result along these lines, due to Gundy
and Stein in 1979 (see~\cite{GS}), concerns the operator
$F(\Delta_1, \Delta_2)$ where $F: {\mathbb
R^2}\rightarrow{\mathbb C}$ is a given bounded function and
$\Delta_i=-\sum_{k=1}^{n_i} \partial_{x_k}^2$ is the standard
Laplace operator on the Euclidean space ${\mathbb R}^{n_i}$,
for $i=1$,~2. This operator $F(\Delta_1, \Delta_2)$ is
initially defined by Fourier analysis on $L^2({\mathbb
R}^{n_1}\times {\mathbb R}^{n_2})$, and extends to a bounded
operator on $L^p({\mathbb R}^{n_1}\times {\mathbb R}^{n_2})$
for all $p\in (1, \infty)$ provided the function $F$ satisfies
\begin{eqnarray}\label{Marcinkiewcz}
    \left| \partial_{\lambda_2}^{\beta}
        \partial_{\lambda_1}^{\alpha} F(\lambda_1, \lambda_2) \right|
    \lesssim \lambda_1^{-\alpha}\lambda_2^{-\beta}
        \  \ {\text{for all }} \alpha\leq \alpha_0 {\text{ and }} \beta\leq\beta_0,
\end{eqnarray}
for some sufficiently large $\alpha_0$ and $\beta_0$. The proof
is based on the idea of pointwise majorization of $F(\Delta_1,
\Delta_2)$ by the Littlewood-Paley product $g$ and
$g_{\lambda}$ functions. One cannot obtain ($H^1$, $L^1$)
boundedness without requiring the function~$F$ to have a
considerable degree of smoothness.

The complicated structure of Hardy spaces and BMO on product
spaces is illustrated by the counterexample of Carleson
\cite{C}, which disproved the rectangle atomic decomposition
conjectures about those spaces. Nevertheless, due to the
boundedness criterion established by R. Fefferman \cite{F} and
by Journ\'e~\cite{J2}, in order to prove the ($H^1$, $L^1$)
estimates for $L^2$-bounded linear operators it suffices to
work on the rectangle atoms (we explain the details in
Sections~3 and 5). Using this criterion L.K.~Chen~\cite{Ch}
proved several results on multiplier operators, one of which
asserts that if
\[
    \left| \partial_{\lambda_2}^{\beta} \partial_{\lambda_1}^{\alpha} F(\lambda_1, \lambda_2)  \right|
    \lesssim \lambda_1^{-\alpha}\lambda_2^{-\beta}  \  \ \ \ {\rm for   \  } \alpha
    \leq \left[\frac{n_1}{2}\right]+1, \ \beta\leq \left[\frac{n_2}{2}\right]+1,
\]
then $F(\Delta_1, \Delta_2)$ is bounded on $L^p({\mathbb
R}^{n_1}\times {\mathbb R}^{n_2})$ for all $p\in (1, \infty)$.

{\textbf {ii)}}  For multivariable spectral multipliers $F(L_1,
L_2)$, Sikora proved a multiplier theorem (see \cite[Theorem
2.1]{Si1}) under the assumption that the integral kernels
$p_t(x_i,y_i)$, $i = 1$, 2, of the semigroups~$e^{-tL_i}$
satisfy Gaussian estimates.
He showed that if $F: {\mathbb R}^2\rightarrow {\mathbb C}$ is
a continuous function satisfying the following estimate on its
Sobolev norms:
\begin{eqnarray}\label{e1.2}\hspace{1cm}
    \sup_{t>0} \| \eta(\lambda_1 + \lambda_2)
        F(t\lambda_1, t\lambda_2)\|_{W^{s, \infty}(\mathbb R^2)}
    < \infty\ \ \ \ \ {\text{ for some } }s>(n_1+n_2)/2,
\end{eqnarray}
then the operator $F(L_1, L_2)$ is of weak type $(1,1)$, and
hence by interpolation and duality, $F(L_1, L_2)$ is bounded on
$L^p( {X_1\times X_2} ) $ for all $1<p<\infty$. Here $X_i$, $i
= 1$, 2, is a metric space equipped with a doubling measure,
$n_i$ is the exponent in the doubling condition
(see~\eqref{doubling} below) and $\eta\in
C_0^{\infty}(0,\infty)$ is a nonzero auxiliary cut-off
function.

Note that the condition~\eqref{e1.2} on $F$ is similar to the
condition in the H\"ormander--Mikhlin Fourier multiplier result
(see for example~\cite{Ho, St1, St2}). It was noted on page~322
of~\cite{Si1} that it would be interesting to obtain a version
of~\cite[Theorem 2.1]{Si1} with condition \eqref{e1.2} replaced
by condition~\eqref{Marcinkiewcz}.
This problem was addressed in \cite{DLY}, as explained in (iii)
below. In this paper, we will obtain results in this direction
for multivariable spectral multipliers in general settings.

{\textbf {iii)}} \  In \cite{DLY}, the second, third and fifth
authors of this paper addressed Marcinkiewicz-type multipliers
on Hardy and Lebesgue spaces on ${\mathbb R^{n_1}\times
{\mathbb R^{n_2}}}$. More precisely, they considered
nonnegative self-adjoint operators $L_i$ on $L^2({\mathbb
R}^{n_i})$, $i = 1$, 2, satisfying Gaussian estimates. One can
define a class of product Hardy spaces $H_{L_1, L_2}^1(
{{\mathbb R}^{n_1}\times {\mathbb R}^{n_2}})$ associated with
such operators $L_1, L_2$ (see \cite{DLY}). Suppose the
following mixed Sobolev norm on functions is defined on
${\mathbb R}\times {\mathbb R}:$
\begin{eqnarray}\label{e1.4}
    \|F\|_{W^{(s_1,s_2), 2}(\mathbb R\times {\mathbb R})}
    := \left\|\left( I-\frac{\partial^2}{\partial \lambda^2_1}\right)^{s_1/2}
    \left( I-\frac{\partial^2}{\partial \lambda^2_2}\right)^{s_2/2}
        F(\lambda_1, \lambda_2)\right\|_{ L^{2}(\mathbb R^2)}.
\end{eqnarray}
With this norm, if
\begin{eqnarray}\label{e1.7}
    \sup_{t_1, t_2>0}
        \left\|\eta(\lambda_1)\eta(\lambda_2)
        F(t_1\lambda_1, t_2\lambda_2) \right\|_{W^{(s_1,s_2), 2}
            (\mathbb R\times {\mathbb R})}
    < \infty,
\end{eqnarray}
\begin{eqnarray}\label{e1.5}
    \sup_{t_1>0}\left\|\eta(\lambda_1)  F(t_1\lambda_1, \lambda_2)
    \right\|_{W^{(s_1,s_2), 2}(\mathbb R\times {\mathbb R})}
    < \infty,
\end{eqnarray}
and
\begin{eqnarray}\label{e1.6}
    \sup_{t_2>0}\left\| \eta(\lambda_2) F( \lambda_1, t_2\lambda_2)
    \right\|_{W^{(s_1,s_2), 2}(\mathbb R\times {\mathbb R})}
    < \infty
\end{eqnarray}
for some $s_1 > (n_1 + 1)/2, s_2 > (n_2 + 1)/2$, then the
operator $F( L_1, L_2)$ extends to a bounded operator from
$H_{L_1, L_2}^1( {{\mathbb R}^{n_1}\times {\mathbb R}^{n_2}} )$
to $L^1( {{\mathbb R}^{n_1}\times {\mathbb R}^{n_2}} )$, and
hence by interpolation and duality, $F(L_1,L_2)$ is bounded on
$L^p({{\mathbb R}^{n_1}\times {\mathbb R}^{n_2}})$ for all $1 <
p < \infty$.

In this paper we continue the line of research described above.
We build a theory of Marcinkiewicz-type multipliers, in the
sense of condition~\eqref{e1.7}, which applies in a rather
general setting of self-adjoint operators and metric measure
spaces satisfying the doubling condition~\eqref{doubling}. In
particular, these spaces are of homogeneous type. We reduce the
number of derivatives required from $s_i > (n_i + 1)/2$ to $s_i
> n_i/2$, by replacing the assumption of Gaussian estimates by
the finite propagation speed property (see
Subsection~\ref{sec:finitespeed}) and Plancherel or restriction
type estimates (see Subsection~\ref{sec:restriction}).

In the strategy adopted in this work, we observe that
Plancherel or restriction type estimates play a crucial role.
Restriction type estimates originate from classical Fourier
analysis where one considers the so-called restriction problem:
\emph{describe all pairs of exponents $(p,2)$ such that the
restriction operator
\begin{eqnarray*}
    R_\lambda(f)(\omega)
    = \widehat{f}(\lambda \omega)
\end{eqnarray*}
is bounded from $L^p({\mathbb R}^n) \to L^2({\bf S}^{n-1})$.}
Here $\widehat{f}$ is the Fourier transform of $f$ and
$\omega\in {\bf S}^{n-1}$ is a point on the unit sphere. The
classical Stein--Tomas restriction theorem establishes the
$L^p\to L^2$ boundedness of the restriction operator
$R_\lambda$ for $1\leq p\leq (2n+3)/(n+1)$; see~\cite{St2}.

The restriction operator $R_\lambda$ is fundamentally linked to
the standard Laplacian~$\Delta = \sum_{i = 1}^n
\partial_{x_i}^2$ on~$\RR^n$, as follows. Let $d
E_{\sqrt{-\Delta}}(\lambda)$ be the spectral measure of
$\sqrt{-\Delta}$.
A short calculation shows that
\begin{eqnarray}\label{e1.8}
    d E_{\sqrt{-\Delta}}(\lambda)
    = \frac{\lambda^{n-1}}{(2\pi)^n} R_\lambda^*R_\lambda,
\end{eqnarray}
(compare \cite{GHS}). %
%
By a standard $T^*T$ argument, $ d E_{\sqrt{-\Delta}}(\lambda)$
is bounded from $L^p({\mathbb R}^n)$ to $L^{p'}({\mathbb
R}^n)$, where $p'$ is the conjugate exponent of $p$ ($1/p +
1/p' = 1$). In other words, the Stein--Tomas restriction
estimates on the $L^p\to L^2$ norm of $R_\lambda$ can be
expressed purely in terms of spectral resolutions of the
self-adjoint operator $-\Delta$.

Motivated by these considerations for the standard Laplace
operator~$-\Delta$ on~$\mathbb{R}^n$, \emph{restriction type
estimates} ${\rm (ST^{2 }_{p, 2})}$ (described below) for
general self-adjoint operators~$L$ on metric measure spaces~$X$
were introduced in~\cite{COSY}. An important step in obtaining
the boundedness of spectral multiplier operators is to prove
restriction type estimates~${\rm (ST^{2 }_{p, 2})}$. Some
results in this direction were obtained in~\cite{COSY, DOS,
GHS}.


In preparation for stating our main results, we describe our
setting for this paper. We consider two nonnegative
self-adjoint operators $L_i$, $i = 1$, 2, on $L^2(X_i)$ where
$(X_i, d_i, \mu_i)$ is a metric measure space, $d_i$ is a
metric and the measure $\mu_i$ satisfies the \emph{doubling
condition}: there exist positive constants $C$, $n_i$ such that
\begin{eqnarray}\label{doubling}
    V_i(x_i, \lambda r)
    \le C \lambda^{n_i} V_i(x_i,   r)\,\;
    {\text{for all}}\, x_i \in X_i,\, \lambda \ge 1, \, r > 0.
\end{eqnarray}
Here $V_i(x_i,r)$ denotes the volume of the open ball
$B_i(x_i,r)$ of center $x_i$ and radius $r$.
In particular, $X_i$ is a space of homogeneous type, as defined
by Coifman and Weiss~\cite[Chapter~III]{CW}. For simplicity,
however, in this paper we assume that $d_i$ is a metric, not
just a quasimetric.

In addition, we assume that the operators $L_i$, $i = 1$, 2,
satisfy the finite propagation speed property for the
corresponding wave equation. With all these assumptions in
place, one can define a class of product Hardy spaces $H_{L_1,
L_2}^p(X_1\times X_2)$, $1\leq p\leq 2$, associated with the
operators $L_1$ and $L_2$ (see~\cite{CDLWY}).

Following~\cite{COSY}, we say that the operators $L_i$, $i =
1$, 2, satisfy the \emph{Stein--Tomas restriction type
estimates} $({\rm{ST}}_{p,2}^q)$ if for every $R > 0$ and all
Borel functions $F$ such that $\support F\subseteq [0,R]$,
\[
    \|F(\sqrt{L_i})P_{B_i(x_i,r)}\|_{L^p(X_i)\to L^2(X_i)}
    \leq CV_i(x_i,r)^{-(1/p-1/2)}(Rr)^{n_i(1/p-1/2)}
        \|F(R\lambda)\|_{L^q({\mathbb R})}
    \leqno{({\rm{ST}}_{p,2}^q)}
\]
for all $x_i\in X_i$ and all $r\geq 1/R$.

We note that if the volume is polynomial, in other words, if
$V_i(x_i,r) \sim r^{n_i}$, then ${({\rm ST}^2_{p, 2})}$ is
equivalent to the $(p,2)$ restriction estimate of Stein--Tomas
(see Section~2 below).

Our main results on spectral multipliers are as follows. (They
are restated as Theorems~\ref{th3.1} and~\ref{th3.2} below.)
They involve restriction type estimates of the
form~$({\rm{ST}}_{p_i,2}^2)$ and~$({\rm{ST}}_{p_i,2}^\infty)$
respectively, along with corresponding Sobolev conditions.

\begin{mainresult}\label{mainresult1}
    Let $(X_1,d_1,\mu_1)$ and $(X_2,d_2,\mu_2)$ be metric
    measure spaces satisfying the doubling
    condition~\eqref{doubling} with exponents~$n_1$, $n_2$
    respectively. Let $L_1$ and $L_2$ be nonnegative
    self-adjoint operators acting on $L^2(X_1)$ and $L^2(X_2)$
    respectively. Suppose $L_1$ and $L_2$ satisfy the finite
    propagation speed property as well as restriction type estimates
    $({\rm{ST}}_{p_1,2}^2)$ and~$({\rm{ST}}_{p_2,2}^2)$
    respectively, for some $p_1$, $p_2$ with $1\leq p_1 < 2$,
    $1\leq p_2 < 2$. Suppose $s_1 > n_1/2$ and $s_2 > n_2/2$.
    Let $F$ be a bounded Borel function satisfying the
    following Sobolev conditions with respect to the $L^2$
    norm: the mixed Sobolev condition
    \begin{eqnarray}\label{e1.11}
        \sup_{t_1, t_2>0}
            \left\|\eta(\lambda_1)\eta(\lambda_2)
            F(t_1\lambda_1, t_2\lambda_2) \right\|_{W^{(s_1,s_2),2}
            (\mathbb R\times {\mathbb R})}
        < \infty,
    \end{eqnarray}
    and the two one-variable Sobolev conditions
    \begin{eqnarray}\label{e1.9}
        \sup_{t_1>0} \left\|\eta(\lambda_1)  F(t_1\lambda_1, 0)
            \right\|_{W^{s_1, 2}(\mathbb R)}
        < \infty,
    \end{eqnarray}
    and
    \begin{eqnarray}\label{e1.10}
        \sup_{t_2>0}\left\| \eta(\lambda_2) F(0, t_2\lambda_2)
        \right\|_{W^{s_2, 2}(\mathbb R)}
        < \infty,
    \end{eqnarray}
    where $\eta\in C_0^\infty(0,\infty)$ is a nonzero
    auxiliary cut-off function. Then the operator $F(L_1, L_2)$
    extends to a bounded operator from $H_{L_1,L_2}^1(
    {X_1\times X_2} )$ to $L^1( {X_1\times X_2} )$. Moreover,
    the operator  $F(L_1, L_2)$ is bounded on $L^p(X_1\times
    X_2)$ for all $p\in (p_{\max},p_{\max}')$, where
    $p_{\max} = \max\{p_1,p_2\}$.
\end{mainresult}

\begin{mainresult}\label{mainresult2}
    Let $(X_1,d_1,\mu_1)$ and $(X_2,d_2,\mu_2)$ be metric
    measure spaces satisfying the doubling
    condition~\eqref{doubling} with exponents~$n_1$, $n_2$
    respectively. Let $L_1$ and $L_2$ be nonnegative
    self-adjoint operators acting on $L^2(X_1)$ and $L^2(X_2)$
    respectively. Suppose that $L_1$ and $L_2$ satisfy the
    finite propagation speed property as well as restriction
    type estimates $({\rm{ST}}_{p_1,2}^2)$
    and~$({\rm{ST}}_{p_2,2}^2)$ respectively, for some $p_1$,
    $p_2$ with $1\leq p_1 < 2$, $1\leq p_2 < 2$. Suppose $s_1 >
    n_1/2$ and $s_2 > n_2/2$. Let $F$ be a bounded Borel
    function satisfying the following Sobolev conditions with
    respect to the $L^\infty$ norm: the mixed Sobolev condition
    \begin{eqnarray}\label{e1.14}
        \sup_{t_1, t_2>0}
            \left\|\eta(\lambda_1)\eta(\lambda_2)
            F(t_1\lambda_1, t_2\lambda_2)\right\|_{W^{(s_1,s_2), \infty}
                (\mathbb R\times {\mathbb R})}
        < \infty,
    \end{eqnarray}
    and the two one-variable Sobolev conditions
    \begin{eqnarray}\label{e1.12}
        \sup_{t_1 > 0}
            \left\|\eta(\lambda_1) F(t_1\lambda_1, 0)
                \right\|_{W^{s_1, \infty}(\mathbb R)}
        < \infty,
    \end{eqnarray}
    and
    \begin{eqnarray}\label{e1.13}
        \sup_{t_2 > 0}
            \left\| \eta(\lambda_2)
                F(0, t_2\lambda_2)\right\|_{W^{s_2, \infty}(\mathbb R)}
        < \infty,
    \end{eqnarray}
    where $\eta\in C_0^\infty(0,\infty)$ is a nonzero auxiliary
    cut-off function. Then the operator $F(L_1, L_2)$ extends
    to a bounded operator from $H_{L_1,L_2}^1({X_1\times X_2}
    )$ to $L^1( {X_1\times X_2} )$. Moreover, the operator
    $F(L_1, L_2)$ is bounded on $L^p(X_1\times X_2)$ for all
    $p\in (p_{\max},p_{\max}')$, where
    $p_{\max} = \max\{p_1,p_2\}$.
\end{mainresult}

\begin{remark}
    If the operator $F(L_1, L_2)$ is bounded from
    $H_{L_1,L_2}^1({X_1\times X_2} )$ to $L^1( {X_1\times X_2}
    )$, then by the interpolation theorem (Theorem~\ref{th2.7}
    below), we know that $F(L_1, L_2)$ is bounded from
    $H_{L_1,L_2}^p({X_1\times X_2} )$ to $L^p( {X_1\times X_2}
    )$ not just for $p\in (p_{\max},2]$, but for all $p\in
    (1,2]$. However, Proposition~\ref{prop2.3} below, which
    guarantees that the Hardy spaces $H_{L_1,L_2}^p({X_1\times
    X_2} )$ coincide with $L^p( {X_1\times X_2} )$, relies on
    the assumption that $p\in (p_{\max},2]$. Thus we obtain the
    $L^p$ boundedness only for all $p\in (p_{\max},2]$ and by
    duality for all $p\in (p_{\max},p_{\max}')$.
%
\end{remark}

\begin{remark}
    Note that if $L_i$, $i = 1$, 2, are nonnegative
    self-adjoint operators on $L^2(X_i)$ satisfying Gaussian
    estimates, then estimates ${\rm (FS)}$ and $({\rm  ST}_{1,
    2}^{\infty})$ hold (see Section~2.2 below). Hence for any
    bounded Borel function satisfying conditions \eqref{e1.14},
    \eqref{e1.12} and \eqref{e1.13} for some $s_1> n_1/2$,
    $s_2>n_2/2$, the operator $F( L_1, L_2)$ extends to a
    bounded operator from $H_{L_1, L_2}^1( {X_1\times X_2} )$
    to $L^1( {X_1\times X_2} )$. It follows by interpolation
    and duality that $F( L_1, L_2)$ is bounded on $L^p(
    {X_1\times X_2} ) $ for all $p$ with $1 < p < \infty$. See
    Corollary~\ref{coro3.3} below.
\end{remark}

\begin{remark}\label{rem:removeconditionbrief}
    If we assume in addition that the spectrum of $L_1$ (or
    $L_2$) does not include the point~$0$, or that 0~belongs to
    the absolutely continuous spectrum of $L_1$ (or $L_2$),
    then the conclusions of Main Results~1.1 and~1.2 still hold
    without assuming the one-variable Sobolev
    conditions~\eqref{e1.10} (or \eqref{e1.9})
    and~\eqref{e1.13} (or~\eqref{e1.12}). For more detail see
    Remark~\ref{rem:removecondition} at the end of
    Section~\ref{sec:proofofmainthms}. For example, when $L_1$
    and $L_2$ are standard Laplacians on Euclidean spaces, the
    mixed Sobolev condition~\eqref{e1.11} alone on~$F$ is
    enough to yield the $L^p$-boundedness of $F(L_1,L_2)$ for
    all $1<p<\infty$. We note that in this setting of
    Laplacians on Euclidean spaces, under the same
    condition~\eqref{e1.11} Carbery and Seeger proved the $H^1
    \to H^1$ boundedness of the multiplier, using a different
    approach (see~\cite[Corollary~5.2]{CaSe}).
\end{remark}

\begin{remark}
    We note that our methods could be extended to establish
    $k$-parameter analogues of our results, for the spectral
    multiplier~$F(L_1,\ldots , L_k)$ with $k \geq 3$.
\end{remark}

Our main contribution in this paper is to draw together three
disparate components in order to establish the desired
boundedness results for the Marcinkiewicz-type spectral
multiplier~$F(L_1, L_2)$ on product spaces. These three
components are: (a)~Hardy spaces $H^1_{L_1,L_2}(X_1\times X_2)$
associated to operators $L_1$, $L_2$, and their atomic
decomposition, (b)~two-parameter restriction type estimates,
and their use in proving off-diagonal estimates, and (c)~the
use of Journ\'e's Lemma in the proofs of our main results.

We should mention that unlike the one-parameter case, under the
hypotheses of Main Result~1.1 or~1.2 we cannot expect the
operator $F(L_1,L_2)$ in the setting of $X_1 \times X_2$ to
satisfy the weak-type $(1,1)$ estimate, since, as shown
in~\cite{DLY}, even in the setting of Euclidean
spaces~$\mathbb{R}^{n_1}\times\mathbb{R}^{n_2}$ the
weak-type~$(1,1)$ estimate does not hold for $F(L_1,L_2)$. That
is, there does not exist a uniform constant $C$ such that for
every $f\in L^1(\mathbb{R}^{n_1}\times\mathbb{R}^{n_2})$, the
endpoint estimate
\[
    \big| \big\{(x_1, x_2)\in \mathbb{R}^{n_1}\times\mathbb{R}^{n_2}:
        |Tf(x)| > \alpha\big\}\big|
    \leq \frac{C}{\alpha}\|f\|_{L^1(\mathbb{R}^{n_1}\times\mathbb{R}^{n_2})},
    \ \ \ {\text {for all}}\, \alpha > 0
\]
holds. From the point of view of interpolation theory, Main
Results~1.1 and~1.2 show that the Calder\'on--Zygmund theory on
product domains shifts the focus of attention from the `weak'
$L^1$ theory to the `strong' $(H^1, L^1)$-theory (see for
example \cite{BLYZ, CYZ, CF1, CF2, DLY, F}). The recently developed atomic
decomposition for the Hardy spaces $H^1_{L_1,L_2}(X_1\times
X_2)$ (see Section~3.1 below and~\cite{CDLWY}) plays a key role
in the proofs of our main results in
Section~\ref{sec:proofofmainthms}.

A novel aspect of this paper is the two-parameter restriction
type estimates obtained in Lemmas~\ref{le4.1} and~\ref{le4.2};
to our knowledge these have not appeared in the literature
before. They play a crucial role in the chain of implications
from the one-parameter restriction type estimates $({\rm
ST}^{2}_{p, 2})$ and~$({\rm ST}^{\infty}_{p, 2})$ to our
off-diagonal estimates.

In the multiparameter setting, we make use of Journ\'e's
Lemma~\cite{J1}, which acts as a multiparameter substitute for
the usual one-parameter covering lemmas.

We note two important ways in which our paper is not simply a
straightforward generalization of~\cite{DLY}. First, here we
prove the boundedness of the spectral multiplier assuming the
existence of fewer derivatives ($n_i/2$ instead of $(n_i +
1)/2$ derivatives). Second, for a large class of operators, we
show that the only condition required on the bounded Borel
function~$F$ is the mixed Sobolev condition~\eqref{e1.14} (as
discussed in Remarks~\ref{rem:removeconditionbrief}
and~\ref{rem:removecondition}). Each of these points requires
numerous innovations in the proofs.

\medskip

The layout of the paper is as follows. In
Section~\ref{sec:notationprelim} we give our notation and
recall some basic properties of heat kernels, the notion of the
finite propagation speed property for the wave equation, and
the Stein--Tomas restriction type estimates. In
Section~\ref{sec:hardy} we recall the theory of the Hardy
spaces $H^p_{L_1,L_2}(X_1\times X_2)$, $1\leq p\leq 2$,
associated with operators $L_1$, $L_2$. Our main results,
Theorems~\ref{th3.1} and~\ref{th3.2}, are stated in
Section~\ref{sec:multipliertheorems}, along with some examples
of operators for which the assumptions of Theorems~\ref{th3.1}
and~\ref{th3.2} hold, namely the standard Laplace operator,
Schr\"odinger operators with inverse-square potential, and
sub-Laplacians on homogeneous groups. In
Section~\ref{sec:off-diagonal} we use the Stein--Tomas
restriction type estimates to obtain off-diagonal estimates of
multivariable spectral multipliers, which play an important
role in the proofs of Theorems~\ref{th3.1} and~\ref{th3.2} in
Section~\ref{sec:proofofmainthms}. Finally, in
Section~\ref{sec:applications} we conclude our paper by
applying our results to the analysis of second-order elliptic
differential operators, specifically Riesz-transform-like
operators and double Bochner--Riesz means.

Throughout, the symbols ``$c$" and ``$C$" will denote (possibly
different) constants that are independent of the essential
variables.

\medskip

\section{Notation and preliminary
results}\label{sec:notationprelim}
\setcounter{equation}{0}

We begin by describing our notation and basic assumptions. As
is usually the case in the theory of spectral multipliers, we
require the doubling condition for the underlying spaces,
together with some estimates for semigroups generated by
operators (see for instance~\cite{DOS, DY, Ou, Si1}). We assume
that the ambient spaces $X_i$, $i = 1$, 2, are each equipped
with a Borel measure~$\mu_i$ and metric~$d_i$. We suppose
throughout that for $i = 1$, 2, $\mu_i$ is a doubling measure:
there exists a constant $C_i>0$ such that
\begin{eqnarray}
V_i(x_i,2r)\leq C_i V_i(x_i,r)\ \ \ \ \ \
{\text{for all}}\,r>0,\,x_i\in X_i,\
\label{e2.1}
\end{eqnarray}
where $V_i(x_i,r)$ denotes the volume of the open ball $B_i =
B_i(x_i,r) := \{y_i \in X_i: d_i(y_i, x_i)<r\}$
(see~\cite{CW}).
%
%
Equivalently, there exist positive constants $C$ and~$n_i$ such
that the doubling condition~\eqref{doubling} holds:
\[
    V_i(x_i, \lambda r)
    \le C \lambda^{n_i} V_i(x_i,   r)\,\;
    {\text{for all}}\, x_i \in X_i,\, \lambda \ge 1, \, r > 0.
\]
%
Throughout this paper we always assume that $V_i(X_i) = \infty$
and condition~\eqref{doubling} holds.

In particular, $X_1$ and $X_2$ are spaces of homogeneous type,
as defined in~\cite[Chapter~III]{CW}. We recall that a
\emph{space of homogeneous type} is a set~$X$ equipped with a
quasimetric~$\rho$ for which all the balls $B(x,r)$ are open,
and with a nonnegative Borel measure~$\mu$ that is a
\emph{doubling measure}:
\[
    \mu(B(x,2r))
    \leq A_1 \mu(B(x,r))
    \qquad
    \text{for all~$x\in X$, $r > 0$},
\]
where $A_1$ is finite and independent of $x$ and~$r$. Also the
volume $\mu(B(x,r))$ of each ball is required to be finite. A
quasimetric satisfies the same conditions as a metric, except
that the triangle inequality is weakened to the condition
\[
    \rho(x,y)
    \leq A_0 \big(\rho(x,z) + \rho(z,y)\big)
    \qquad
    \text{for all~$x$, $y$, $z\in X$,}
\]
where $A_0$ is finite and independent of $x$, $y$ and~$z$.

Given $\lambda > 0$ and a ball~$B_i$ in~$X_i$, we write
$\lambda B_i$ for the $\lambda$-dilated ball, which is the ball
with the same center as $B_i$ and with radius $r_{\lambda B_i}
= \lambda r_{B_i}$. We define the sets
\begin{equation}
    U_0(B_i)
    := B_i
    \qquad {\rm and} \qquad
    U_k(B_i)
    :=2^k B_i\backslash 2^{k-1}B_i
    \,\,\mbox{ for }\,\,k = 1,2, \dots.
    \label{e2.3}
\end{equation}

Let $X_1\times X_2$ be the Cartesian product of $X_1$ and~$X_2$
with the product measure $\mu = \mu_1\times \mu_2$. For
functions $f$ defined on $X_1 \times X_2$, we will use the
mixed Lebesgue norms $\|\cdot\|_{L^{p_1}({X_1; \, L^{p_2}(X_2))
}}$ and  $\|\cdot\|_{L^{p_1}(X_2; \, L^{p_2}(X_1)) }$ for
$p_1,p_2\in [1,\infty]$ defined by
$$
 \|f\|_{L^{p_1}({X_1; \, L^{p_2}(X_2)) }}
 :=\left\| \big\|f(x_1, x_2)\big\|_{L^{p_2}(X_2)} \right\|_{L^{p_1}(X_1)}
$$
and $$
 \|f\|_{L^{p_1}(X_2; \, L^{p_2}(X_1)) }
 :=\left\| \big\|f(x_1, x_2)\big\|_{L^{p_2}(X_1)} \right\|_{L^{p_1}(X_2)},\ \  
$$
respectively. In particular, when $p_1=p_2=p$ for some $1\leq
p\leq \infty$, we see that $\|f\|_{L^{p_1}(X_1; \,
L^{p_2}(X_2)) } =\|f\|_{L^{p_1}(X_2; \, L^{p_2}(X_1))
}=\|f\|_{L^p(X_1\times X_2)}$.

For functions $F$ defined on $\mathbb R\times \mathbb R$, to
avoid ambiguity, we use the notation
$$
 \|F\|_{L^{p_1}_{\lambda_1}({\mathbb R; \, L_{\lambda_2}^{p_2}(\mathbb R)) }}
 :=\left\| \big\|F(\lambda_1, \lambda_2)\big\|_{L_{\lambda_2}^{p_2}(\mathbb R)} \right\|_{L_{\lambda_1}^{p_1}(\mathbb R)}
$$
and
$$
 \|F\|_{L^{p_1}_{\lambda_2}(\mathbb R; \, L^{p_2}_{\lambda_1}(\mathbb R)) }
 :=\left\| \big\|F(\lambda_1, \lambda_2)\big\|_{L_{\lambda_1}^{p_2}(\mathbb R)} \right\|_{L_{\lambda_2}^{p_1}(\mathbb R)}.\ \  
$$
We use the following mixed Sobolev norms:
$$
    \|F\|_{W^{(s_1,s_2), p}(\mathbb R\times {\mathbb R})}
    := \left\|\left( I-\frac{\partial^2}{\partial \la^2_1}\right)^{s_1/2}
        \left(I - \frac{\partial^2}{\partial \la^2_2}\right)^{s_2/2}F
        \right\|_{ L^{p}(\mathbb R^2)} \, ,
$$
$$
    \|F\|_{L^{p_2}_{\lambda_2}({\mathbb R}; \,
        W^{s,p_1}_{\lambda_1}(\mathbb R))}
    := \left\|\left(I - \frac{\partial^2}{\partial \la^2_1}\right)^{s/2}F
        \right\|_{{L^{p_2}_{\lambda_2}({\mathbb R};
        \, L^{p_1}_{\lambda_1}({\mathbb R})) }} \, ,
$$
and
$$
    \|F\|_{L^{p_1}_{\lambda_1}({\mathbb R};
        \, W^{s,p_2}_{\lambda_2}(\mathbb R))}
    := \left\|\left(I-\frac{\partial^2}{\partial \la^2_2}\right)^{s/2}F
        \right\|_{{L^{p_1}_{\lambda_1}({\mathbb R};
        \, L^{p_2}_{\lambda_2}({\mathbb R})) }}. 
$$

We write $\|T\|_{\mathcal{B}_1 \to \mathcal{B}_2}$ for the
operator norm of a bounded linear operator $T: \mathcal{B}_1
\to \mathcal{B}_2$ between Banach spaces $\mathcal{B}_1$
and~$\mathcal{B}_2$.


Given a subset $E\subseteq X_1\times X_2$, we denote
by~$\chi_E$ the characteristic function of~$E$ and set
$$
    P_Ef(x)
    := \chi_E(x) f(x).
$$

Given $R > 0$ and a one-variable function $F$ on ${\mathbb R}$,
the dilation $\delta_{R}F$ is defined by $\delta_R F(\lambda)
:= F(R\la)$. Given $R_1$, $R_2 > 0$ and a function $F$
on~${\mathbb R^2}$, the dilation $\delta_{(R_1, R_2)}F$ is
defined by
\begin{eqnarray}\label{dilation}
    (\delta_{(R_1, R_2)}F)(\lambda_1, \lambda_2)
    := F(R_1\lambda_1,R_2\lambda_2).
\end{eqnarray}
In particular, $\delta_{(R_1, 1)}F:=F(R_1\lambda_1,\lambda_2)$
represents dilation in the first variable only, by a factor
of~$R_1$, and similarly for $\delta_{(1, R_2)}F :=
F(\lambda_1,R_2\lambda_2)$.
For $F\in L^2({\mathbb R^2})$, the Fourier transform of $F$ is
given by
\begin{eqnarray*}
    {\widehat F} (\xi_1, \xi_2)
    &:=& \frac{1}{(2\pi)^2} \int_{-\infty}^{+\infty}\int_{-\infty}^{+\infty}
    F(\lambda_1, \lambda_2)   e^{-i(\xi_1\lambda_1+\xi_2\lambda_2)}\, d\lambda_1d\lambda_2.
\end{eqnarray*}

As usual, $p'$ denotes the conjugate exponent of
$p\in[1,\infty]$, so that $1/p + 1/{p'} = 1$.

\subsection{Finite propagation speed property for the wave
equation}\label{sec:finitespeed} \ For $i = 1$, 2, and for $\rho
> 0$, we define the diagonal subset~$\D^i_\rho$ of $X_i\times
X_i$ by
\begin{equation*}
    \D^i_\rho
    :=\{ (x_i,\, y_i)\in X_i\times X_i: {d}_i(x_i,\, y_i) \le \rho \}.
\end{equation*}
Consider an operator $T_i$ from $L^p(X_i)$ to $L^q(X_i)$, for some
$p$, $q$. Fix $\rho > 0$. We say that
\begin{equation}\label{e2.4}
    \support K_{T_i}
    \subseteq \D^i_\rho
\end{equation}
if $\langle T_i f, g \rangle = 0$ for all $f$, $g$ in~$C(X_i)$ whose
supports are far apart with respect to~$\rho$, in the following
sense:
\begin{enumerate}
    \item[(i)] $f$ is supported in some ball
        $B_i(x_i^{(f)}, r^{(f)})$,

    \item[(ii)] $g$ is supported in some ball
        $B_i(x_i^{(g)}, r^{(g)})$, and

    \item[(iii)] $r^{(f)} + r^{(g)} + \rho <
        d_i(x_i^{(f)},x_i^{(g)})$.
\end{enumerate}
Note that if $T_i$ is an integral operator with a kernel $K_{T_i}$,
then~\eqref{e2.4} coincides with the standard meaning of $\support
K_{T_i} \subseteq \D^i_\rho$, namely $K_{T_i}(x_i, \, y_i) = 0$ for
all $(x_i, \, y_i) \notin \D^i_\rho$.

\begin{definition}\label{def:finitspeed}
    For $i = 1$, 2, let $L_i$ be a nonnegative self-adjoint
    operator on~$L^2({X_i})$, where $X_i$ is a metric measure
    space. We say that the operators $L_i$, $i = 1$, 2, satisfy
    the \emph{finite propagation speed property} if
    \[
        \support K_{\cos(t\sqrt{L_i})}
        \subseteq \D^i_t \quad {\text{for all $t > 0$.}}
        \leqno{\rm (FS)}
    \]
\end{definition}

Property (FS) holds for most second-order self-adjoint
operators. It is equivalent to the \emph{Davies--Gaffney
estimates}, which state that the semigroup $e^{-tL_i}$ has
integral kernels $p_t(x_i,y_i)$ satisfying the following
estimates
\[
    \leqno{\rm(DG)} \hspace{4cm}
    \big\|P_{B_i(x_i, \sqrt{t})} e^{-tL_i}
        P_{B_i(y_i, \sqrt{t})}\big\|_{L^2(X_i)\to L^2(X_i)}
    \leq C \exp\left(-c \dfrac{d_i^2(x_i, y_i)}{t}\right)
\]
for all $t > 0$ and all $x_i$, $y_i\in X_i$. See for
example~\cite{CGT, CoSi, Si2}.




\medskip

Recall that the semigroup $e^{-tL_i}$ generated by $L_i$, $i =
1$, 2, is said to satisfy \emph{Gaussian estimates} if there
exist constants $C$, $c > 0$ such that for all $t > 0$ and all
$x_i$, $y_i\in X_i$, the semigroup $e^{-tL_i}$ has integral
kernels $p_t(x_i,y_i)$ satisfying the following estimates
\[
    \leqno{\rm (GE)} \hspace{4cm}
    |p_t(x_i, y_i)|
    \leq \dfrac{C}{V_i(x_i,\sqrt{t})}
        \exp\left(-c \frac{d_i^2(x_i,y_i)}{t} \right).
\]
It is immediate that ${\rm (GE)} \Rightarrow {\rm (DG)}$.
However, there are many operators which satisfy Davies--Gaffney
estimates ${\rm (DG)}$ but for which classical Gaussian
estimates ${\rm (GE)}$ fail. For example, Schr\"odinger
operators with inverse-square potential fall into this
category~\cite{LSV}. See also
Section~\ref{sec:multipliertheorems} below.

\subsection{Stein--Tomas restriction type estimates}
\label{sec:restriction}

In this subsection we recall the restriction type estimates
$({\rm ST}^{q}_{p, 2})$, which were originally introduced
in~\cite{DOS} for $p = 1$, and then in~\cite{COSY} for
general~$p$. These estimates control the $L^p\to L^2$ operator
norm of the composition of a multiplier $F(\sqrt{L_i})$ with a
suitable projection, in terms of the $L^q(\RR)$ norm of a
suitable dilate of~$F$.

\begin{definition}\label{def:RTE}
    Let $X_i$, $i = 1$, 2, be metric measure spaces satisfying
    the doubling condition~\eqref{doubling} with
    exponent~$n_i$. Consider nonnegative self-adjoint operators
    $L_i$, $i = 1$, 2, and numbers $p$, $q$ such that $1 \leq p
    < 2$ and $1 \leq q \leq \infty$. We say that $L_i$
    satisfies the \emph{Stein--Tomas restriction type
    estimates}~$({\rm ST}^{q}_{p, 2})$ if there is a
    constant~$C > 0$ such that for each $R > 0$ and for all
    Borel functions $F$ with $\support F \subset [0, R]$, we
    have
    \[
        \big\|F(\sqrt{L_i})P_{B_i(x_i, r)} \big\|_{L^p(X_i)\to L^2(X_i)}
        \leq C \, \bigg(\frac{(Rr)^{n_i}}{ V_i(x_i, r)}\bigg)^{1/p - 1/2}
            \big\|F(R\lambda)\big\|_{L^q({\mathbb R})}
        \leqno{({\rm ST}^{q}_{p, 2})}
    \]
    for all $x_i\in X_i$ and for all $r\geq 1/R$.
\end{definition}

Note that if restriction type estimates $({\rm ST}^{q}_{p, 2})$
hold for some $q\in [1, \infty)$, then $({\rm ST}^{\tilde
q}_{p, 2}) $ holds for all $\tilde{q}\geq q$ including the case
$\tilde{q} = \infty$.

It is known that for the standard Laplace operator $\Delta_i $
on~${\mathbb R^{n_i}}$, restriction type estimates $({\rm
ST}^{2}_{p, 2})$ are equivalent to the $(p,2)$ restriction
estimates of Stein--Tomas, namely
\[
    \big\|dE_{\sqrt{\Delta_i}}(\lambda)\big\|_{L^p(\RR^{n_i})\to L^{p'}(\RR^{n_i})}
    \leq C\lambda^{n_i(1/p- 1/p')-1}
\]
for all $p\in [1, 2(n+1)/(n+3)]$; see~\cite[Proposition
2.4]{COSY}.

The following result, which was proved
in~\cite[Proposition~2.3]{COSY}, shows that if ${q} = \infty$
then restriction type estimates~$({\rm ST}^{{\infty}}_{p, 2})$
follow from the standard elliptic estimates.

\medskip

\begin{prop}[\cite{COSY}]\label{prop2.33}
    Assume that the metric measure spaces $X_i$, $i = 1$, 2,
    satisfy the doubling condition~\eqref{doubling} with
    exponent~$n_i$. Suppose $1 \leq p < 2 $ and $N > n_i(1/p -
    1/2)$. Let $L_i$, $i = 1$, 2, be nonnegative self-adjoint
    operators. Then ${({\rm ST}^{\infty}_{p, 2})}$ is
    equivalent to each of the following conditions:
    \begin{itemize}
    \item[(a)]  For all $x_i\in X_i$ and  $r\geq t>0$,
    \[
        \big\|e^{-t^2L_i}P_{B_i(x_i, r)}\big\|_{L^p(X_i)\to L^2(X_i)}
        \leq CV_i(x_i, r)^{1/2 - 1/p}
            \left(\frac{r}{t}\right)^{n_i(1/p - 1/2)}.
        \leqno({\rm G}_{p,2})
    \]

    \item[(b)] For all $x\in X$ and $r\geq  t > 0$,
    \[
        \big\|(I+t\sqrt{L_i})^{-N}P_{B_i(x_i, r)}\big\|_{L^p(X_i)\to L^2(X_i)}
        \leq CV_i(x_i, r)^{1/2 - 1/p}
            \left(\frac{r}{t}\right)^{n_i(1/p - 1/2)}.
        \leqno({\rm E}_{p,2})
    \]
    \end{itemize}
\end{prop}

Note that when $L_i$, $i = 1$, 2, are nonnegative self-adjoint
operators on $L^2(X_i)$ satisfying Gaussian estimates~${\rm
(GE)}$, it follows from Proposition~\ref{prop2.33} that
estimate $({\rm  ST}_{1, 2}^{\infty})$ holds.


\medskip

\section{The Hardy space $H^{p}_{L_1, L_2}(X_1\times X_2)$}
\label{sec:hardy}
\setcounter{equation}{0}

The theory of the Hardy spaces $H^{1}_{L_1, L_2}({\mathbb
R^{n_1}}\times {\mathbb R^{n_2}})$ associated to operators
$L_i$, $i = 1$, 2, was introduced and studied in \cite{DSTY,
DLY} in the case where the underlying space ${\mathbb
R^{n_1}}\times {\mathbb R^{n_2}}$ is the product of Euclidean
spaces and the operators $L_i$, $i = 1$, 2, are self-adjoint
and nonnegative and possess Gaussian upper bounds on their heat
kernels. This theory was generalized in~\cite{CDLWY} to the
case where the underlying space $X_1\times X_2$ is the product
of metric measure spaces $X_i$, $i = 1$, 2, satisfying the
doubling condition~\eqref{doubling} with exponent $n_i$, and
$L_i$, $i = 1$, 2, are nonnegative self-adjoint operators
satisfying the finite propagation speed property. We give a
brief presentation of this generalization, which is needed for
our study of multivariable spectral multipliers. For more
detail, see~\cite{CDLWY}.

Given a function $f$ on $L^2({X_1\times X_2})$, the square
function $S(f)$ associated with operators $L_1$ and $L_2$ is
defined by
\begin{eqnarray}\label{e2.6}
\hskip.7cm S(f)(x):= \left(\iint_{\Gamma(x) }\big|\big(t_1^2L_1e^{-t_1^2L_1}\otimes
t_2^2L_2e^{-t_2^2L_2}\big)f(y)\big|^2\
{dy \ \! dt\over t_1t_2V_1(x_1,t_1) V_2(x_2,t_2)}\right)^{1/2}, \label{e3.222}
\end{eqnarray}
where $\Gamma(x)$ is the product cone
$\Gamma(x):=\Gamma_1(x_1)\times\Gamma_2(x_2)$ and
$\Gamma_i(x_i):=\{(y_i,t_i)\in X_i\times
\mathbb{R}_+:d_i(x_i,y_i)<t_i\}$, $i = 1$, 2. It is known that
there exist constants $C_1, C_2$ with $0<C_1\leq C_2<\infty$
such that for any $f\in L^2({X_1\times X_2})$,
\begin{eqnarray}\label{e2.7}
    C_1\|f\|_{L^2(X_1\times X_2)}
    \leq \|S(f)\|_{L^2(X_1\times X_2)}
    \leq C_2\|f\|_{L^2(X_1\times X_2)}.
\end{eqnarray}

\begin{definition} \label{def2.2}
    Suppose $1\leq p\leq 2$. The \emph{Hardy space
    $H^p_{{L_1,L_2}}({X_1\times X_2})$ associated to $L_1$ and
    $L_2$} is defined as the completion of the set
    \[
        \{ f\in L^2({X_1\times X_2}): \|S(f)\|_{L^p({X_1\times X_2})}<\infty \}
    \]
    with respect to the norm
    \[
    \|f\|_{H^{p}_{{L_1,L_2}}({X_1\times X_2})
    } :=\|S(f) \|_{L^p({X_1\times X_2})}.
    \]
\end{definition}

Using Fubini's theorem and the spectral theorem, it can be
shown that $H^{2}_{{L_1,L_2}}({X_1\times X_2}) = L^2({X_1\times
X_2})$ with equivalent norms. Additionally, the set
$H^{1}_{{L_1,L_2}}({X_1\times X_2})\cap L^2({X_1\times X_2})$
is dense in $H^{1}_{{L_1,L_2}}({X_1\times X_2})$. Note that in
the special case of ${\mathbb R}^{n_1}\times {\mathbb R}^{n_2}$
with $L_i$ being the usual Laplacian $\Delta_{n_i}$ on
${\mathbb R^{n_i}}$, the Hardy space
$H^p_{{L_1,L_2}}({X_1\times X_2})$ from Definition~\ref{def2.2}
coincides with the Hardy space $H^1({\mathbb R^{n_1}}\times
{\mathbb R^{n_2}})$, with equivalent norms; see~\cite{CF1,
CF2}.

\begin{prop}\label{prop2.3}
    Let  $L_i$, $i = 1$, 2, be nonnegative self-adjoint operators
    on $L^2(X_i)$ satisfying the finite propagation speed
    property ${\rm (FS)}$ and the estimate $({\rm G}_{p_i,2})$
    for some $p_i$ with $1\leq p_i < 2$. Then for each $p$ with
    $\max\{p_1, p_2\} < p \leq 2$, the Hardy space $H^{p}_{L_1,
    L_2}(X_1\times X_2)$ and the Lebesgue space $L^p(X_1\times
    X_2)$ coincide and their norms are equivalent.
\end{prop}

\begin{proof}
This proof is an extension of a similar proof of the
corresponding result for the one-parameter Hardy space
$H^p_L(X)$ (see \cite[Proposition 9.1(v)]{HMMc} and~\cite{Au}).
For the proof of Proposition~\ref{prop2.3}, we refer the reader
to~\cite{CDLWY}.
\end{proof}

%
%
%
%

\subsection[Atomic decomposition for $H^{1}_{L_1, L_2}(X_1\times X_2)$]{Atomic
decomposition for the Hardy space $H^{1}_{L_1, L_2}(X_1\times
X_2)$} \label{sec:Hardyatomic} Following \cite{CDLWY}, we
introduce the notion of $(H^1_{L_1, L_2}, 2, N)$-atoms
associated to operators $L_1$ and~$L_2$.
As noted in the introduction, our metric measure spaces
$(X_i,d_i,\mu_i)$, $i = 1$, 2, are examples of spaces of
homogeneous type. Thus the following result of M.~Christ
(see~\cite{Ch1} and the references therein) shows that they
possess a dyadic grid analogous to that of Euclidean space.

\begin{theorem}[\cite{Ch1}, Theorem~11]\label{le5.2}
Let $(X,\rho,\mu)$ be a space of homogeneous type (as defined
in Section~\ref{sec:notationprelim}). Then there exist a
collection of open subsets $\{I_{\alpha}^k\subset{ X}:\,k\in
{\mathbb Z},\alpha\in \Lambda_k\}$, where $\Lambda_k$ denotes
some (possibly finite) index set depending on $k$, and
constants $\delta\in (0,1)$, $a_0 > 0$, $\eta > 0$ and $C_1$,
$C_2 < \infty$ such that
\begin{itemize}
\item[(i)] \ $\mu({ X}\backslash\cup_{\alpha}I_{\alpha}^k)
    = 0$ for all $k\in{\mathbb Z}$.

\item[(ii)] \ If $\ell\geq k$ then either
    $I_{\beta}^{\ell}\subset I_{\alpha}^{k}$ or
    $I_{\beta}^{\ell}\cap I_{\alpha}^{k} = \emptyset$.

\item[(iii)]\ For each $(k,\alpha)$ and each $\ell < k$,
    there is a unique $\beta$ such that $I_{\alpha}^k
    \subset I_{\beta}^\ell$.

\item[(iv)]\ Diameter $(I_{\alpha}^k)\leq C_1\delta^k$.

\item[(v)]\ Each $I_{\alpha}^k$ contains some ball
    $B(z^k_{\alpha},a_0\delta^k)$.

    \item[(vi)] $\mu(x\in I_\alpha^k \, : \, \rho(x,
        X\setminus I_\alpha^k) \leq t \delta^k) \leq C_2
        t^\eta \mu(I_\alpha^k) \qquad \text{for all $k$,
        $\alpha$, for all $t > 0$.}$
\end{itemize}
\end{theorem}

The sets $I_\alpha^k$ are known as dyadic cubes. We write
$\ell(I_\alpha^k)$ for the diameter of~$I_\alpha^k$. By the
\emph{$s$-fold dilate $sI_\alpha^k$ of a dyadic
cube~$I_\alpha^k$} for $s > 0$, we mean the ball
$B(z^k_\alpha,sC_1\delta^k/2)$ centred at $z^k_\alpha$ of
radius~$sC_1\delta^k/2$.

Let $X_i$, $i = 1$, 2, be spaces of homogeneous type. The open
set $I_{\alpha_1}^{k_1}\times I_{\alpha_2}^{k_2}$, for $k_1$,
$k_2\in \mathbb{Z}$, $\alpha_1\in \Lambda_{k_1}$ and
$\alpha_2\in \Lambda_{k_2}$, is called a \emph{dyadic rectangle
of $X_1\times X_2$}. Let $\Omega\subset X_1\times X_2$ be an
open set of finite measure. Denote by $m(\Omega)$ the maximal
dyadic subrectangles of~$\Omega$. By the \emph{$s$-fold dilate
$sR$ of a dyadic rectangle~$R = I^{k_1}_{\alpha_1}\times
I^{k_2}_{\alpha_2}$} for $s > 0$, we mean the the product $sR
:= sI^{k_1}_{\alpha_1}\times sI^{k_2}_{\alpha_2}$ of the
$s$-fold dilates of the factors.

\begin{definition}\label{def2.5}
Let $N$ be a positive integer. A function $a\in L^2({X_1\times
X_2})$ is called an \emph{$(H^1_{L_1, L_2}, 2, N)$-atom} if it
satisfies the conditions

\smallskip

\noindent $ 1)$  $\support a\subset \Omega$, where $\Omega$ is
an open set of ${X_1\times X_2}$ with finite measure; and

\smallskip
\noindent $ 2)$ $a$ can be further decomposed into
$$
    a = \sum\limits_{R\in m(\Omega)} a_R,
$$
where $m(\Omega)$ is the set of all maximal dyadic
subrectangles of $\Omega$, and there exists a function $b_R \in
L^2({X_1\times X_2})$ such that for each $k_1, k_2\in\{0, 1,
\ldots, N\}$, $b_R$ belongs to the range of $L_1^{k_1}\otimes
L_2^{k_2}$, and moreover such that

\smallskip

 (i) \ \ \ $a_R = \big(L_1^{N} \otimes L_2^{N}\big) b_R$;

 \smallskip

 (ii) \ \ $\support\big(L_1^{k_1}  \otimes L_2^{k_2}\big)b_R\subset
10R$, \ \ $k_1, k_2\in\{0, 1, \ldots, N\}$;

\smallskip

 (iii) \ $||a||_{L^2( {X_1\times X_2} )} \leq
\mu(\Omega)^{-{1\over 2}}$ and
$$
    \sum_{  k_2=0}^N\sum_{  k_2=0}^N \sum_{R\in m(\Omega)}\ell(I_R)^{-4N} \ell(J_R)^{-4N}
    \Big\|\big(\ell(I_R)^2 \, L_1\big)^{k_1}\otimes \big(\ell(J_R)^2 \,
    L_2\big)^{k_2} b_R\Big\|_{L^2({X_1\times X_2})}^2\leq V(\Omega)^{-1}.
$$
\end{definition}

\medskip

We are now able to define an atomic Hardy space,
$H^1_{L_1,L_2,at,N}(X_1\times X_2)$. It is shown
in~\cite{CDLWY} that this atomic Hardy space is equivalent to
the Hardy space $H^1_{L_1,L_2}(X_1\times X_2)$ that is defined
above via square functions.

\medskip

\begin{definition}\label{def2.6}
    Let $N>\max\{n_1/4, n_2/4\}$. The Hardy space $H^1_{L_1, L_2,
    at,N}({X_1\times X_2})$ is defined as follows. We say that $f =
    \sum\lambda_ja_j$ is an \emph{atomic $(H^1_{L_1, L_2}, 2,
    N)$-representation of~$f$} if $\{\lambda_j\}_{j=0}^\infty\in
    \ell^1$, each $a_j$ is an $(H^1_{L_1, L_2}, 2, N)$-atom, and
    the sum converges in $L^2({X_1\times X_2})$. Set
    $$
    \mathbb{H}^1_{L_1, L_2, at, N}({X_1\times X_2}) :=\left\{f: f~\text{has
    an atomic}~(H^1_{L_1, L_2}, 2, N)\text{-representation}\right\},
    $$
    with the norm given by
    \begin{eqnarray}\label{e2.8}
        \hspace*{1cm}
        \|f\|_{\mathbb{H}^1_{L_1, L_2, at, N}({X_1\times X_2})}
        := \inf\left\{ \sum_{j=0}^\infty |\lambda_j|: f = \sum_j\lambda_ja_j
        \ \text{is an atomic}~(H^1_{L_1, L_2}, 2, N)\text{-representation}\right\}.
    \end{eqnarray}
    The \emph{Hardy space} $H^1_{L_1, L_2, at,N}({X_1\times X_2})$ is then
    defined as the completion of $\mathbb{H}^1_{L_1,
    L_2,at,N}({X_1\times X_2})$ with respect to this norm.
\end{definition}

Consequently, when $N > \max\{n_1/4, n_2/4\}$ one may simply
write $H^1_{L_1, L_2, at}(X_1\times X_2)$ in place of
$H^1_{L_1, L_2, at, N}(X_1\times X_2)$, as these spaces are all
equivalent. Thus the Hardy space $H^1_{L_1, L_2}(X_1\times
X_2)$ is given by
$$
    H^1_{L_1, L_2}(X_1\times X_2)
    = H^1_{L_1, L_2, at}(X_1\times X_2)
    = H^1_{L_1, L_2, at, N}(X_1\times X_2)
$$
for each $N > \max\{n_1/4, n_2/4\}$.


\subsection{An interpolation theorem}
\ Finally, we state the following Marcinkiewicz-type
interpolation theorem. For its proof, we refer the reader
to~\cite{CDLWY}.

\begin{theorem}\label{th2.7}
Let $L_i$, $i = 1$, 2, be nonnegative self-adjoint operators on
$L^2(X_i)$ satisfying the finite propagation speed
property~${\rm (FS)}$.
Let $T$ be a sublinear operator bounded from $H^1_{L_1,
L_2}({X_1\times X_2})$ into $L^1({X_1\times X_2})$ and bounded
on $L^2({X_1\times X_2})$ with operator norms $C_1$ and $C_2$,
respectively. If $1 < p \leq 2$, then $T$ is bounded from
$H^p_{L_1, L_2}({X_1\times X_2})$ into $L^p({X_1\times X_2})$
and
\begin{equation}\label{e2.9}
    \|Tf\|_{H^p_{L_1,L_2}({X_1\times X_2})}
      \leq C \|f\|_{L^p({X_1\times X_2})},
\end{equation}
where $C$ depends only on $C_1, C_2$, and $p$.
\end{theorem}

For more detail on Hardy spaces associated to operators in the
one-parameter and multiparameter settings, see also~\cite{AMR,
CDLWY, DSTY, DLY, DY2, HLMMY, HM, HMMc} and the references
therein.

\medskip

\section{Multivariable spectral multiplier
theorems}\label{sec:multipliertheorems}
\setcounter{equation}{0}

In this section we state our main results, which show that
restriction type estimates together with the finite propagation
speed property can be used to obtain Marcinkiewicz-type
multiplier theorems on Hardy and Lebesgue spaces. We defer the
proofs to Section~\ref{sec:proofofmainthms}.
We assume that the metric measure spaces $X_i$ satisfy the
doubling condition~\eqref{doubling} with exponent~$n_i$, for $i
= 1$, 2.

In what follows, we shall be working with a nontrivial
auxiliary function $\phi$ with compact support. Let $\phi$ be a
$C_0^\infty(\mathbb{R})$ function such that
\begin{eqnarray}\label{e3.1}
    \support \phi \subseteq (1,4)
    \ \ \ {\rm and}\ \ \
    \sum_{\ell\in\mathbb{Z}}\phi(2^{-\ell}\lambda) = 1
    \ \ \ {\rm for\ all}\ \ \lambda > 0.
\end{eqnarray}
Set
\begin{eqnarray}\label{e3.2}
    \eta_{1}(\lambda_1)
    := \phi(\lambda_1),
    \ \  \eta_{2}( \lambda_2)
    := \phi(\lambda_2)
    \ \ \ {\rm and}\ \ \
    \eta_{(1,2)}(\lambda_1, \lambda_2)
    := \phi(\lambda_1)\phi(\lambda_2).
\end{eqnarray}

The aim of this paper is to prove the following two
Marcinkiewicz-type theorems, which involve restriction type
estimates of the form $({\rm{ {ST}}}_{p_i,2}^{2})$ and~$({\rm{
{ST}}}_{p_i,2}^{\infty})$ respectively, together with
corresponding Sobolev conditions.

\begin{theorem}\label{th3.1}
    Let $X_i$, $i = 1$, 2, be metric measure spaces satisfying
    the doubling condition~\eqref{doubling} with
    exponent~$n_i$. Let $L_i$, $i = 1$, 2, be nonnegative
    self-adjoint operators on $L^2(X_i)$ satisfying the finite
    propagation speed property ${\rm (FS)}$ and restriction
    type estimates $({\rm{{ST}}}_{p_i, 2}^{2})$ for some $p_i$
    with $1\leq p_i<2$. Suppose $s_1 > n_1/2$ and $s_2 >
    n_2/2$. Let $F$ be a bounded Borel function satisfying the
    following Sobolev conditions with respect to the $L^2$
    norm: the mixed Sobolev condition
    \begin{eqnarray}\label{e3.5}
        \sup_{t_1, t_2>0} \big\|\eta_{(1,2)}\delta_{(t_1, t_2)}F
            \big\|_{W^{(s_1,s_2),2}(\mathbb R\times {\mathbb R})}
            < \infty,
    \end{eqnarray}
    and the two one-variable Sobolev conditions
    \begin{eqnarray}\label{e3.3}
        \sup_{t_1>0} \big\|\eta_{1}\delta_{(t_1, 1)}F(\cdot,0)
        \big\|_{W^{s_1, 2}(\mathbb R)}
        < \infty,
    \end{eqnarray}
    and
    \begin{eqnarray}\label{e3.4}
        \sup_{t_2>0} \big\|\eta_{2}\delta_{(1, t_2)}F(0,\cdot)
        \big\|_{W^{s_2, 2}(\mathbb R)}
        < \infty.
    \end{eqnarray}
    Then

    \smallskip

    \begin{itemize}
        \item[(i)] the operator $F( L_1, L_2)$ extends to a
            bounded operator from $H_{L_1, L_2}^1(
            {X_1\times X_2} )$ to $L^1( {X_1\times X_2} )$,
            and

        \smallskip

        \item[(ii)]  $F( L_1, L_2)$ is bounded on $L^p(
            {X_1\times X_2} ) $ for all $p\in
            (p_{\max},p_{\max}')$, where $p_{\max} :=
            \max\{p_1,p_2\}$.
    \end{itemize}
\end{theorem}


\begin{theorem}\label{th3.2}
    Let $X_i$, $i = 1$, 2, be metric measure spaces satisfying
    the doubling condition~\eqref{doubling} with
    exponent~$n_i$. Let $L_i$, $i = 1$, 2, be nonnegative
    self-adjoint operators on $L^2(X_i)$ satisfying the finite
    propagation speed property ${\rm (FS)}$ and restriction
    type estimates $({\rm{{ST}}}_{p_i,2}^\infty)$ for some
    $p_i$ with $1 \leq p_i < 2$.  Suppose $s_1 > n_1/2$ and
    $s_2
    > n_2/2$. Let $F$ be a bounded Borel function satisfying the
    following Sobolev conditions with respect to the $L^\infty$
    norm: the mixed Sobolev condition
    \begin{eqnarray}\label{e3.8}
        \sup_{t_1, t_2 > 0} \big\|\eta_{(1,2)}\delta_{(t_1, t_2)}F
        \big\|_{W^{(s_1,s_2),\infty}(\mathbb R\times {\mathbb R})}
        < \infty,
    \end{eqnarray}
    and the two one-variable Sobolev conditions
    \begin{eqnarray}\label{e3.6}
        \sup_{t_1 > 0} \big\|\eta_{1}\delta_{(t_1, 1)}F(\cdot,0)
        \big\|_{W^{s_1, \infty}(\mathbb R)}
        < \infty,
    \end{eqnarray}
    and
    \begin{eqnarray}\label{e3.7}
        \sup_{t_2>0} \big\|\eta_{2}\delta_{(1, t_2)}F(0,\cdot)
        \big\|_{W^{s_2, \infty}(\mathbb R)}
        < \infty.
    \end{eqnarray}
    Then
    \smallskip
    \begin{itemize}
        \item[(i)] the operator $F( L_1, L_2)$ extends to a
            bounded operator  from $H_{L_1, L_2}^1( {X_1\times
            X_2} )$ to $L^1( {X_1\times X_2} )$, and

        \smallskip

        \item[(ii)]  $F( L_1, L_2)$ is bounded on $L^p(
            {X_1\times X_2} ) $ for all $p\in
            (p_{\max},p_{\max}')$, where $p_{\max} :=
            \max\{p_1,p_2\}$.
    \end{itemize}
\end{theorem}

\medskip

As we have seen in Section~2.2, in the important special case
when $L_i$, $i = 1$, 2, are nonnegative self-adjoint operators
on $L^2(X_i)$ satisfying Gaussian estimates ${\rm (GE)}$, it
follows from Proposition~\ref{prop2.33} that estimates ${\rm
(FS)}$ and $({\rm  ST}_{1, 2}^{\infty})$ hold. Therefore in
this case one can omit the hypotheses of the finite propagation
speed property and restriction type estimates in
Theorem~\ref{th3.2}. We describe the details in
Corollary~\ref{coro3.3} below.

\medskip

\begin{corollary}\label{coro3.3}
    Let $X_i$, $i = 1$, 2, be metric measure spaces satisfying
    the doubling condition~\eqref{doubling} with
    exponent~$n_i$. Let $L_i$, $i = 1$, 2, be nonnegative
    self-adjoint operators on $L^2(X_i)$ satisfying Gaussian
    estimates ${\rm (GE)}$. Additionally, assume that $F$ is a
    bounded Borel function satisfying conditions \eqref{e3.8},
    \eqref{e3.6} and \eqref{e3.7} for some $s_1 > n_1/2$, $s_2
    > n_2/2$. Then the operator $F( L_1, L_2)$ extends to a
    bounded operator from $H_{L_1, L_2}^1( {X_1\times X_2} )$
    to $L^1({X_1\times X_2} )$. Hence by interpolation and
    duality, $F(L_1, L_2)$ is bounded on $L^p( {X_1\times
    X_2})$ for all $p$ such that $1 < p < \infty$.
\end{corollary}

\begin{proof}
Corollary~\ref{coro3.3} follows from
Theorem~\ref{th3.2} and Proposition~\ref{prop2.3}.
\end{proof}

\medskip

The proofs of Theorems~\ref{th3.1} and ~\ref{th3.2} are given
in Section~\ref{sec:proofofmainthms}; they rely on the
off-diagonal estimates established in
Section~\ref{sec:off-diagonal}.

\smallskip

We end this section by describing three types of nonnegative
self-adjoint operators~$L$ that satisfy property~(FS) together
with restriction type estimates $({\rm ST}^{2}_{p, 2})$ for
some $p$ with $1 \leq p < 2$.

\medskip

\noindent {\bf  1) Standard Laplace operator.} As mentioned in
the introduction, when $\Delta$ denotes the usual Laplacian on
${\mathbb R}^n$, $n\geq 2$, we have
\begin{equation}\label{e3.9}
    \big\|d E_{\sqrt{\Delta}}(\lambda)\big\|_{L^p(\RR^n) \to L^{p'}(\RR^n) }
    \leq C\lambda^{n(1/p - 1/(p')) - 1}, \ \ \ \ \lambda>0
\end{equation}
for $1\leq p \leq 2(n+1)/(n+3)$. Thus the restriction type
estimates ${({\rm ST}^{2}_{p, 2})} $ for the standard Laplace
operator on ${\mathbb R}^n$ are valid for $1\leq p\leq
{2(n+1)/(n+3)}.$

\medskip

\noindent {\bf  2)\ Schr\"odinger operators with inverse-square
potential.} Consider an inverse square potential $V(x) = c/| x
|^2$. Fix $n
> 2$ and assume that $c > -{(n-2)^2/4}$. Define $L := -\Delta +
V$ on $L^2(\RR^n, dx)$ by the quadratic form method. The
classical Hardy inequality
\begin{equation}\label{hardy1}
    - \Delta
    \geq  {(n-2)^2\over 4}|x|^{-2}
\end{equation}
shows that for all $c > -{(n-2)^2/4}$, the self-adjoint
operator $L$ is nonnegative. If $c \ge 0$, then the semigroup
$\exp(-tL)$ is pointwise bounded by the Gaussian semigroup and
hence acts on all $L^p$ spaces with $1 \le p \le \infty$. For
$c < 0$, set $p_c^{\ast} = n/\sigma$ where $\sigma =
(n-2)/2-\sqrt{(n - 2)^2/4 + c}$\, ; then $\exp(-tL)$ acts as a
uniformly bounded semigroup on $L^p(\RR^n)$ for $ p \in
((p_c^{\ast})', p_c^{\ast})$, and the range $((p_c^{\ast})',
p_c^{\ast})$ is optimal (see for example~\cite{LSV}).

It is known (see for instance~\cite{COSY}) that $L$ satisfies
restriction type estimates~${({\rm ST}^{2 }_{p, 2})} $ for all
$p\in [1,2n/(n + 2)]$ if $c \geq 0$, and for all $p \in
((p_c^{\ast})', 2n/(n + 2)]$ if $c < 0$.


\medskip

\noindent {\bf  3) Sub-Laplacians on homogeneous groups.} Let
${\bf G}$ be a Lie group of polynomial growth and let $X_1,
\ldots, X_k$ be a system of left-invariant vector fields on
${\bf G}$ satisfying the H\"ormander condition. We define the
Laplace operator $L$ acting on $L^2({\bf G})$ by the formula
\begin{eqnarray}
    L
    := -\sum_{i=1}^k X_i^2.
    \label{e3.12}
\end{eqnarray}
If $B(x,r)$ is the ball defined by the distance associated with
the system $X_1, \ldots, X_k$, then there exist natural numbers
$n_0, n_{\infty}\geq 0$ such that $V(x, r) \sim r^{n_0}$ for
$r\leq 1$ and  $V(x, r)\sim r^{n_{\infty}}$ for $r>1$ (see for
example~\cite[Chapter III.2]{VSC}). It follows that the
doubling condition~\eqref{doubling} holds with exponent $n =
\max\{n_0,{n_{\infty}} \}$, that is,
\[
V(x, \lambda r)
    \le C \lambda^{\max\{n_0,{n_{\infty}}\}} V(x,r) \qquad
    {\text{for all}}\, x \in {\bf G},\, \lambda \ge 1, \, r > 0.
\]

We call ${\bf G}$ a \emph{homogeneous group} if there exists a
family of dilations on ${\bf G}$. A family of dilations on a
Lie group ${\bf G}$ is a one-parameter group $({\tilde\delta
}_t)_{t>0}$ (thus ${\tilde\delta }_t \circ {\tilde\delta }_s =
{\tilde\delta }_{t s}$) of automorphisms of ${\bf G}$
determined by
\[
    {\tilde\delta }_t Y_j = t^{n_j} Y_j ,
\]
where $Y_1, \ldots, Y_{\ell}$ is a linear basis of the Lie
algebra of ${\bf G}$ and $n_j\geq 1$ for $1\leq j\leq \ell$
(see \cite{FS}). We say that an operator $L$  defined by
(\ref{e3.12}) is homogeneous if  ${\tilde\delta }_t X_i = t
X_i$ for $1\leq i\leq k$ and the system $X_1, \ldots, X_k$
satisfies the H\"ormander condition. Then for the
sub-Riemannian geometry corresponding to the system  $X_1,
\ldots, X_k$, one has $n_0 = n_{\infty} = \sum_{j=1}^{\ell}
n_j$ (see \cite{FS}). Hence in a homogeneous group, $n =
\max\{n_0,{n_{\infty}}\} = n_0 = n_{\infty}$. It is well known
that the heat kernel corresponding to the sub-Laplacian~$L$ on
a homogeneous group satisfies Gaussian estimates~${\rm (GE)}$.
It is also not difficult to check that for some constant $C >
0$,
$$
    \big\|F(\sqrt L)\big\|_{L^2(G) \to L^\infty(G)}^2
    = C \int_0^\infty |F(t)|^2 t^{n-1}\,
dt.
$$
See for example \cite[equation~(7.1)]{DOS} or~\cite[Proposition
10]{Ch2}. It follows from the above equality that the
operator~$L$ satisfies estimate ${({\rm ST}^{2 }_{1, 2})}$.

\smallskip

See also Section~\ref{sec:applications} for applications of
Theorems~\ref{th3.1} and~\ref{th3.2} and
Corollary~\ref{coro3.3}.

\medskip

\section{Off-diagonal estimates for multivariable spectral multipliers}
\label{sec:off-diagonal}
\setcounter{equation}{0}

In this section we show that restriction type estimates $({\rm
ST}_{p_i,2}^2)$ or~$({\rm ST}_{p_i,2}^\infty)$ can be used to
obtain off-diagonal estimates on spectral multipliers in the
abstract setting of self-adjoint operators acting on product
metric measure spaces. As usual, we assume that the metric
measure spaces $X_i$, $i = 1$, 2, satisfy the doubling
condition~\eqref{doubling} with exponent~$n_i$. We use these
off-diagonal estimates in the proofs of Theorems~\ref{th3.1}
and~\ref{th3.2}, in Section~\ref{sec:proofofmainthms}.

We begin with two lemmas showing that the estimates $({\rm
ST}_{p_i,2}^2)$ or~$({\rm ST}_{p_i,2}^\infty)$ imply
corresponding two-parameter restriction type estimates.

\begin{lemma}\label{le4.1}
Assume that the nonnegative self-adjoint operators $L_i$, $i =
1$, 2, satisfy restriction type estimates $({\rm{
{ST}}}_{p_i,2}^2)$ for some $p_i$ with $1 \leq p_i < 2$. Then
we have the following two-parameter restriction type estimates:
\begin{itemize}
    \item[(i)] For every $R_1,R_2>0$ and all Borel
        functions $F$ such that $\support F\subseteq
        [0,R_1]\times[0,R_2]$,
        \begin{eqnarray*}
            \big\|F\big(\sqrt{L_1},\sqrt{L_2}\big)
                P_{B_1(x_1,r_1)\times B_2(x_2,r_2)}\big\|_{L^{p_1}(X_1;
                \, L^{p_2}(X_2))\to L^2(X_1\times X_2)}
            \leq C\prod_{i=1}^2 \left({ (R_ir_i)^{n_i}\over  V_i(x_i,r_i))}\right)^{1/p_i - 1/2}
                \big\|\delta_{(R_1,R_2)} F\big\|_{L^2({\mathbb R}^2)}
        \end{eqnarray*}
        for all $x_i\in X_i$ and all $r_i\geq 1/R_i$, $i=1$, 2.

    \item[(ii)] For every $R_1>0$ and all Borel functions $F$
        such that $\support F\subseteq [0,R_1]\times [0,
        \infty)$,
        \begin{eqnarray*}
            \big\|F\big(\sqrt{L_1},\sqrt{L_2}\big)P_{B_1(x_1,r_1)\times X_2}\big\|_{L^{p_1}(X_1;
                \, L^2(X_2))\to L^2(X_1\times X_2)}
            \leq C \left({ (R_1r_1)^{n_1}\over  V_1(x_1,r_1)}\right)^{1/p_1 - 1/2}
                \big\|\delta_{(R_1,1)} F \big\|_{L^{\infty}_{\la_2}({\mathbb R};
                \, L^2_{\la_1}({\mathbb R}))}
        \end{eqnarray*}
        for all $x_1\in X_1$ and all $r_1\geq 1/R_1$.

    \item[(iii)] For every $R_2>0$ and all Borel functions $F$
        such that $\support F\subseteq [0,
        \infty)\times[0,R_2]$,
        \begin{eqnarray*}
            \big\|F\big(\sqrt{L_1},\sqrt{L_2}\big)P_{X_1\times B_2(x_2,r_2)}\big\|_{L^{p_2}(X_2;
                \, L^2(X_1))\to L^2(X_1\times X_2)}
            \leq C\left({  (R_2r_2)^{n_2} \over V_2(x_2,r_2)}\right)^{1/p_2 - 1/2}
                \big\|\delta_{(1,R_2)} F \big\|_{L^{\infty}_{\la_1}({\mathbb R}; \,
                L^2_{\la_2}({\mathbb R}))}
        \end{eqnarray*}
        for all $x_2\in X_2$ and all $r_2\geq 1/R_2$.
\end{itemize}
\end{lemma}

\begin{proof}
Suppose that $F$ is a Borel function with $\support F\subseteq
[0,R_1]\times[0,R_2]$.  For $i = 1$, 2, we consider Riemann
partitions of~$[-1, R_i]$:
$$
-1 = \lambda^{(i)}_1 < \lambda^{(i)}_2 < \ldots < \lambda^{(i)}_{n_i} = R_i, \ \ \ \ \
\Delta^{(i)}_{\ell_i} := \la^{(i)}_{\ell_i+1} - \la^{(i)}_{\ell_i}.
$$
Let $  {\tilde \la}^{(i)}_{\ell_i}\in(\la^{(i)}_{\ell_i},
\la^{(i)}_{\ell_i+1}].$ Taking the limit of the Riemann
approximation (see~\cite[page 310]{Yo}),
\begin{eqnarray*}
    F\big(\sqrt{L_1},\sqrt{L_2}\big)
    &=& \lim_{\substack{ \Delta^{(1)}_{\ell_1}\to 0\\ \Delta^{(2)}_{\ell_2}\to 0}}
        \sum_{\ell_1=1}^{n_1} \sum_{\ell_2=1}^{n_2}
        F({\tilde \la}^{(1)}_{\ell_1} ,{\tilde \la}^{(2)}_{\ell_2})
        \chi_{(\la^{(1)}_{\ell_1}, \la^{(1)}_{\ell_1 +1}  ]}\big(\sqrt{L_1}\big)
        \chi_{(\la^{(2)}_{\ell_2}, \la^{(2)}_{\ell_2 +1}  ]}\big(\sqrt{L_2}\big),
\end{eqnarray*}
which gives
\begin{eqnarray} \label{e4.1}
    &&\big\|F\big(\sqrt{L_1},\sqrt{L_2}\big)
        P_{B_1(x_1,r_1)\times B_2(x_2,r_2)}g\big\|^2_{L^2(X_1\times X_2)}\nonumber \\
    &&\hskip.5cm\leq \lim_{\substack{ \Delta^{(1)}_{\ell_1}\to 0\\
        \Delta^{(2)}_{\ell_2}\to 0}}
        \sum_{\ell_1=1}^{n_1} \sum_{\ell_2=1}^{n_2}
        |F({\tilde \la}^{(1)}_{\ell_1} ,{\tilde\la}^{(2)}_{\ell_2})|^2 \
        \big\|\chi_{(\la^{(1)}_{\ell_1},\la^{(1)}_{\ell_1+1} ]}\big(\sqrt{L_1}\big)
        \chi_{(\la^{(2)}_{\ell_2}, \la^{(2)}_{\ell_2+1} ]}\big(\sqrt{L_2}\big)
        P_{B_1(x_1,r_1)\times B_2(x_2,r_2)}g\big\|_{L^2(X_1\times X_2)}^2.
\end{eqnarray}
We then apply the restriction type estimates
$({\rm{{ST}}}_{p_i,2}^2)$ for $L_i$ to obtain
\begin{eqnarray*}
    &&\big\|\chi_{(\la^{(1)}_{\ell_1}, \la^{(1)}_{\ell_1+1}  ]}\big(\sqrt{L_1}\big)
        \chi_{(\la^{(2)}_{\ell_2}, \la^{(2)}_{\ell_2+1}  ]}\big(\sqrt{L_2}\big)
        P_{B_1(x_1,r_1)\times B_2(x_2,r_2)}
        \big\|_{L^{p_1}(X_1; \, L^{p_2}(X_2))\to L^2(X_1\times X_2)} \\
    &&\hskip.5cm\leq
        \prod_{i=1}^2
        \big\|\chi_{(\la^{(i)}_{\ell_i}, \la^{(i)}_{\ell_i+1} ]}
        \big(\sqrt{L_i}\big) P_{B_i(x_i,r_i)}
        \big\|_{L^{p_i}(X_i)\to L^2(X_i)} \\
    &&\hskip.5cm\leq  C\prod_{i=1}^2\left(\frac{(R_ir_i)^{n_i}}{{ V(x_i,r_i) }}\right)^{1/p_i-1/2}
        \sqrt{\frac{\la^{(i)}_{\ell_i + 1} - \la^{(i)}_{\ell_i}}{R_i}}.
\end{eqnarray*}
Substituting this estimate back into inequality~\eqref{e4.1}
and taking the limit, we obtain item~(i).

Turning to item~(ii), suppose that $F$ is a Borel function with
$\support F\subseteq [0,R_1]\times [0, \infty)$. Consider
Riemann partitions of~$[-1, R_1)\times [-1, \infty)$:
$$
    -1 = \lambda^{(1)}_1 < \lambda^{(1)}_2 < \cdots
        < \lambda^{(1)}_{n_1} = R_1, \ \ \
    -1 = \lambda^{(2)}_1 < \lambda^{(2)}_2 < \cdots
        < \lambda^{(2)}_{\ell_2} < \lambda^{(2)}_{\ell_2+1} < \cdots .
$$
For every $i = 1$, 2, we set $ \Delta^{(i)}_{\ell_i} :=
\la^{(i)}_{\ell_i+1} - \la^{(i)}_{\ell_i} $ for $\ell_1 = 1, 2,
\ldots, n_1$ and $\ell_2 = 1, 2, \ldots$. Let $ {\tilde
\la}^{(i)}_{\ell_i}\in(\la^{(i)}_{\ell_i},
\la^{(i)}_{\ell_i+1}].$ Again, taking the limit of the Riemann
approximation, an argument as in item~(i) shows that
\begin{eqnarray} \label{e4.111}
    &&\big\|F\big(\sqrt{L_1},\sqrt{L_2}\big)P_{B_1(x_1,r_1)\times X_2}g
        \big\|^2_{L^2(X_1\times X_2)} \nonumber\\
    &&\hskip.5cm\leq \lim_{\substack{ \Delta^{(1)}_{\ell_1}\to 0\\
        \Delta^{(2)}_{\ell_2}\to 0}}
        \sum_{\ell_1=1}^{n_1} \sum_{\ell_2=1}^{\infty}
        |F({\tilde \la}^{(1)}_{\ell_1} ,{\tilde \la}^{(2)}_{\ell_2})|^2 \
        \big\|\chi_{(\la^{(1)}_{\ell_1}, \la^{(1)}_{\ell_1+1} ]}\big(\sqrt{L_1}\big)
        \chi_{(\la^{(2)}_{\ell_2}, \la^{(2)}_{\ell_2+1} ]}\big(\sqrt{L_2}\big)
        P_{B_1(x_1,r_1)\times X_2}g\big\|_{L^2(X_1\times X_2)}^2.
\end{eqnarray}
We then apply the restriction type estimates
$({\rm{{ST}}}_{p_1,2}^2)$ for $L_1$ and Minkowski's inequality
to obtain
\begin{eqnarray*}
    && \big\|\chi_{(\la^{(1)}_{\ell_1}, \la^{(1)}_{\ell_1+1}  ]}\big(\sqrt{L_1}\big)
      \chi_{(\la^{(2)}_{\ell_2}, \la^{(2)}_{\ell_2+1}  ]}\big(\sqrt{L_2}\big)
      P_{B_1(x_1,r_1)\times X_2} f\big\|^2_{  L^2(X_1\times X_2)} \\
      &&\hskip.5cm \leq
      \big\|\chi_{(\la^{(1)}_{\ell_1},
      \la^{(1)}_{\ell_1+1}  ]}\big(\sqrt{L_1}\big) P_{B_1(x_1,r_1)} \big\|^2_{L^{p_1}(X_1)\to L^2(X_1)} \
      \big\|\chi_{(\la^{(2)}_{\ell_2}, \la^{(2)}_{\ell_2+1}  ]}\big(\sqrt{L_2}\big)
      f \big\|^2_{L^{2}(X_2; \, L^{p_1}(X_1)) } \\
    &&\hskip.5cm \leq    C \left({(R_1r_1)^{n_1} \over  { V(x_1,r_1) }}\right)^{2/p_1-1 }
        \frac{\la^{(1)}_{\ell_1 + 1} - \la^{(1)}_{\ell_1}}{R_1} \
        \big\|\chi_{(\la^{(2)}_{\ell_2}, \la^{(2)}_{\ell_2+1}  ]}\big(\sqrt{L_2}\big)f \big\|^2_{L^{p_1}(X_1; \, L^{2}(X_2)) }.
\end{eqnarray*}
This, together with \eqref{e4.111}, yields
\begin{eqnarray*}
    &&\big\|F\big(\sqrt{L_1},\sqrt{L_2}\big)P_{B_1(x_1,r_1)\times X_2}g
        \big\|^2_{L^2(X_1\times X_2)} \\
    &&\hskip.5cm\leq \lim_{\substack{ \Delta^{(1)}_{\ell_1}\to 0\\
        \Delta^{(2)}_{\ell_2}\to 0}}
        \sum_{\ell_2=1}^{\infty}\left(\sum_{\ell_1=1}^{n_1}
        |F({\tilde \la}^{(1)}_{\ell_1} ,{\tilde \la}^{(2)}_{\ell_2})|^2
        \left({(R_1r_1)^{n_1} \over  { V(x_1,r_1) }}\right)^{2/p_1-1 }
        \frac{\la^{(1)}_{\ell_1 + 1} - \la^{(1)}_{\ell_1}}{R_1}
     \right)
     \left(
    \big\|\chi_{(\la^{(2)}_{\ell_2}, \la^{(2)}_{\ell_2+1}  ]}\big(\sqrt{L_2}\big)
        g\big\|_{L^{p_1}(X_1; \, L^{2}(X_2)) }^2\right)\nonumber \\
    &&\hskip.5cm\leq \lim_{ \Delta^{(2)}_{\ell_2}\to 0}
        \left({(R_1r_1)^{n_1} \over {V_1(x_1,r_1) }}\right)^{2/p_1-1}
        \big\|\delta_{(R_1,1)}F\big\|_{L_{\la_2}^\infty(\RR; \, L_{\la_1}^2(\RR))}^2
        \left(\sum_{\ell_2=1}^{\infty}
        \big\| \chi_{(\la^{(2)}_{\ell_2}, \la^{(2)}_{\ell_2+1}  ]}\big(\sqrt{L_2}\big)
        g\big\|_{L^{p_1}(X_1; \, L^{2}(X_2)) }^2\right)\nonumber \\
    &&\hskip.5cm \leq  \left({(R_1r_1)^{n_1} \over  {V_1(x_1,r_1) }} \right)^{2/p_1-1}
        \big\|\delta_{(R_1,1)}F\big\|_{L_{\la_2}^\infty(\RR; \, L_{\la_1}^2(\RR))}^2
        \|g\|_{L^{p_1}(X_1; \, L^2(X_2))}^2,
\end{eqnarray*}
where we have used the fact that
\begin{eqnarray*}
    \left(\sum_{\ell_2=1}^{\infty}
        \big\|\chi_{(\la^{(2)}_{\ell_2}, \la^{(2)}_{\ell_2+1}  ]}\big(\sqrt{L_2}\big)
        g\big\|_{L^{p_1}(X_1, L^{2}(X_2)) }^2\right)
    &\leq& \left(\int_{X_1} \left[  \sum_{\ell_2=1}^{\infty}
        \big\|\chi_{(\la^{(2)}_{\ell_2}, \la^{(2)}_{\ell_2+1}  ]}\big(\sqrt{L_2}\big)
        g\big\|_{ L^{2}(X_2) }^2 \right]^{p_1/2}\right)^{2/p_1}\nonumber \\
    &\leq& C\left(\int_{X_1} \|g\|_{ L^{2}(X_2) }^{p_1}\right)^{2/p_1}\nonumber \\
    &=& C\|g\|_{L^{p_1}(X_1; \, L^2(X_2))}^2
\end{eqnarray*}
by Minkowski's inequality. This proves item~(ii).

Item (iii) can be obtained by a symmetric argument. Hence the
proof of Lemma~\ref{le4.1} is complete.
\end{proof}

\begin{lemma}\label{le4.2}
Assume that the nonnegative self-adjoint operators $L_i$, $i =
1$, 2, satisfy restriction type estimates $({\rm{
{ST}}}_{p_i,2}^\infty)$ for some $p_i$ with $1 \leq p_i < 2$.
Then we have the following two-parameter restriction type
estimates:
\begin{itemize}
    \item[(i)] For every $R_1$, $R_2 > 0$ and all bounded
        Borel functions $F$ such that $\support F\subseteq
        [0,R_1]\times[0,R_2]$,
        \begin{eqnarray*}
            \big\|F\big(\sqrt{L_1},\sqrt{L_2}\big)P_{B_1(x_1,r_1)
                \times B_2(x_2,r_2)}\big\|_{L^{p_1}(X_1; \, L^{p_2}(X_2))\to L^2(X_1\times X_2)}
            \leq C\prod_{i=1}^2 \left({ (R_ir_i)^{n_i}\over  V_i(x_i,r_i)}\right)^{1/p_i-1/2}
                \big\|\delta_{(R_1,R_2)} F\big\|_{L^{\infty}({\mathbb R}^2)}
        \end{eqnarray*}
        for all $x_i\in X_i$ and all $r_i\geq 1/R_i$, $i=1$, 2.

    \item[(ii)] For every $R_1>0$ and all bounded Borel
        functions $F$ such that $\support F\subseteq
        [0,R_1]\times[0, \infty)$,
        \begin{eqnarray*}
            \big\|F\big(\sqrt{L_1},\sqrt{L_2}\big)P_{B_1(x_1,r_1)
                \times X_2}\big\|_{L^{p_1}(X_1; \, L^2(X_2))\to L^2(X_1\times X_2)}
            \leq C \left({ (R_1r_1)^{n_1}\over  V_1(x_1,r_1)}\right)^{1/p_1-1/2}
                \big\|\delta_{(R_1,1)} F\big\|_{L^{\infty}({\mathbb R}^2)}
        \end{eqnarray*}
        for all $x_1\in X_1$ and all $r_1\geq 1/R_1$.

     \item[(iii)] For every $R_2>0$ and all bounded Borel
         functions $F$ such that $\support F\subseteq [0,
         \infty)\times[0,R_2]$,
         \begin{eqnarray*}
            \big\|F\big(\sqrt{L_1},\sqrt{L_2}\big)P_{X_1\times
                B_2(x_2,r_2)}\big\|_{L^{p_2}(X_2; \, L^2(X_1))\to L^2(X_1\times X_2)}
            \leq C\left({  (R_2r_2)^{n_2} \over V_2(x_2,r_2)}\right)^{1/p_2-1/2}
                \big\|\delta_{(1,R_2)} F\big\|_{L^{\infty}({\mathbb R}^2)}
         \end{eqnarray*}
         for all $x_2\in X_2$ and all $r_2\geq 1/R_2$.
\end{itemize}
\end{lemma}

\begin{proof}
Suppose that $F$ is a bounded Borel function with $\support
F\subseteq [0,R_1]\times[0,R_2]$.  We set
\[
    H_1(\la_1) := e^{-\la_1/R_1}\chi_{[0,R_1]}(\la_1), \quad
    H_2(\la_2) := e^{-\la_2/R_2}\chi_{[0,R_2]}(\la_2), \quad\text{and}\quad
    G(\la_1,\la_2) := F(\la_1,\la_2)e^{\la_1/R_1}e^{\la_2/R_2}.
\]
Then we have
$$
    F(\la_1,\la_2)
    = G(\la_1,\la_2)H_1(\la_1)H_2(\la_2).
$$
It is immediate that $\|G\|_{L^\infty(\mathbb R^2)}\leq
C\|F\|_{L^\infty(\mathbb R^2)}$. Since the operators $L_1$ and
$L_2$ satisfy restriction type estimates $({\rm{
{ST}}}_{p_1,2}^\infty)$ and $({\rm{ {ST}}}_{p_2,2}^\infty)$
respectively,  we have
\begin{eqnarray*}
     &&\hspace{-1.8cm}\big\|F\big(\sqrt{L_1},\sqrt{L_2}\big)P_{B_1(x_1,r_1)
        \times B_2(x_2,r_2)}\big\|_{L^{p_1}(X_1; \, L^{p_2}(X_2))\to L^2(X_1\times X_2)} \\
     &= &\bigg\|G\big(\sqrt{L_1},\sqrt{L_2}\big)
        \prod_{i=1}^2 H_i\big(\sqrt{L_i}\big) P_{B_i(x_i,r_i)}
        \bigg\|_{L^{p_1}(X_1; \, L^{p_2}(X_2))\to L^2(X_1\times X_2)}\nonumber\\
     &\leq& \big\|G\big(\sqrt{L_1},\sqrt{L_2}\big)\big\|_{{L^2(X_1\times X_2)\to L^2(X_1\times X_2)}}
        \prod_{i=1}^2\big\|H_i\big(\sqrt{L_i}\big) P_{B_i(x_i,r_i)}
        \big\|_{L^{p_i}(X_i)\to L^2(X_i)}\nonumber\\
     &\leq&  C\|G\|_{L^{\infty}({\mathbb R^2})}
        \prod_{i=1}^2\left[\left({(R_ir_i)^{n_i} \over  { V_i(x_i,r_i) }}\right)^{1/p_i-1/2}
        \big\|\delta_{R_i} H_i\big\|_{L^{\infty}({\mathbb R})}  \right]\nonumber\\
     &\leq&  C\prod_{i=1}^2\left({(R_ir_i)^{n_i} \over  { V_i(x_i,r_i) }}\right)^{1/p_i-1/2}
        \big\|\delta_{(R_1,R_2)} F\big\|_{L^{\infty}({\mathbb R}^2)}\notag,
\end{eqnarray*}
which implies (i) for the pair of operators $(L_1,L_2)$.


The proofs of (ii) and (iii) can be obtained by making minor
modifications to the argument for~(i). We omit the details.
This proves Lemma~\ref{le4.2}.
\end{proof}

Define
$$
    {\mathscr S}_{\rm even}
    := \left\{F:\RR\times\RR\to\CC : F(\la_1,\la_2)~\text{is
        even in each variable separately}
        \right\}.
$$

The aim of this section is to prove Propositions~\ref{prop4.1}
and~\ref{prop4.2}, which play important roles in the proofs of
our multivariable spectral multipliers results,
Theorems~\ref{th3.1} and~\ref{th3.2}, respectively. We begin
with Proposition~\ref{prop4.1}, which involves restriction type
estimates of the form~$({\rm{ {ST}}}_{p_i,2}^{2})$.

\begin{prop}\label{prop4.1}
    Suppose that the metric measure spaces $X_i$, $i = 1$, 2,
    satisfy the doubling condition~\eqref{doubling} with
    exponent~$n_i$. Further assume that $L_i$, $i = 1$, 2, are
    nonnegative self-adjoint operators satisfying the finite
    propagation speed property ${\rm (FS)}$ and restriction
    type estimates $({\rm{ {ST}}}_{p_i,2}^{2})$, for some $p_i$
    with $1 \leq p_i < 2$. Suppose that $F\in {\mathscr S}_{\rm
    even}$.
    \begin{itemize}
        \item[(i)] If the function $F$ is supported on
            $[-R_1,R_1]\times [-R_2,R_2]$, then for each $s
            > 0$ and $t > 0$, for each $\varepsilon > 0$
            there exists a constant $C_{s,t}$ such that for
            every $B_1(x_1,r_1)\times B_2(x_2,r_2)$ and for
            every $j$, $k = 1,2,\ldots,$
        \begin{eqnarray}\label{e4.2}
            \lefteqn{\left\|P_{B_1(x_1,2^jr_1)^c\times B_2(x_,2^kr_2)^c}F\big(\sqrt{L_1},\sqrt{L_2}\big)
            P_{B_1(x_1, r_1)\times  B_2(x_2, r_2)} \right\|_{L^{p_1}(X_1; \, L^{p_2}( X_2))\to L^2(X_1\times X_2)}}
            \hspace{0.2cm}\nonumber\\
            &\leq& C_{s,t}\prod_{i=1}^2\max\left\{{\big( R_ir_i \big)^{n_i}
             \over V_i(x_i,r_i)  }, {1\over {V_i(x_i,R_i^{-1})}}
             \right\}^{1/p_i-1/2}  (2^jr_1 R_1)^{-s} (2^kr_2 R_2)^{-t}
             \big\|\delta_{(R_1,R_2)}F\big\|_{W^{(s+\varepsilon,\, t+\varepsilon), 2} ({\mathbb R\times \mathbb R})};\hspace{-1cm}
        \end{eqnarray}

        \item[(ii)] If the function $F$ is supported on
            $[-R_1,R_1]\times \RR$, then for each $s>0$ and
            for each $\varepsilon>0$, there exists a
            constant $C_s$ such that for every ball
            $B_1(x_1,r_1)$ and for every $j = 1,2,\ldots,$
        \begin{eqnarray}\label{e4.3}
            \lefteqn{\left\|P_{B_1(x_1,2^jr_1)^c\times X_2}F\big(\sqrt{L_1},\sqrt{L_2}\big)P_{B_1(x_1, r_1)\times X_2}
            \right\|_{L^{p_1}(X_1; \, L^2(X_2))\to L^2(X_1\times X_2)}}
            \hspace{0.2cm}\nonumber\\
            &\leq& C_s \max\left\{{\big( R_1r_1 \big)^{n_1} \over V_1(x_1, r_1)  }, {1\over {V_1(x_1,R_1^{-1})}}
            \right\}^{1/p_1-1/2} (2^jr_1R_1)^{-s}
            \big\|\delta_{(R_1,1)}F\big\|_{{L^{\infty}_{\la_2}({\mathbb R}; \, W^{s+\varepsilon,\,  2}_{\la_1}(\mathbb R))}};
        \end{eqnarray}

        \item[(iii)] If the function $F$ is supported on
            $\RR\times [-R_2,R_2]$, then for each $t>0$ and
            for each $\varepsilon > 0$, there exists a
            constant $C_t$ such that for every ball
            $B_2(x_2,r_2)$ and for every $k = 1,2,\ldots,$
        \begin{eqnarray}\label{e4.4}
            \lefteqn{\left\| P_{X_1\times B_2(x_2,2^kr_2)^c}F\big(\sqrt{L_1},\sqrt{L_2}\big)P_{X_1\times B_2(x_2, r_2)}
            \right\|_{L^{p_2}(X_2; \, L^2(X_1))\to L^2(X_1\times X_2)}}
            \hspace{0.2cm}\nonumber\\
            &\leq& C_t \max\left\{{\big( R_2r_2 \big)^{n_2} \over V_2(x_2, r_2)  }, {1\over {V_2(x_2,R_2^{-1})}}
            \right\}^{1/p_2-1/2} (2^kr_2
            R_2)^{-t}\big\|\delta_{(1,R_2)}F\big\|_{L^{\infty}_{\la_1}({\mathbb R}; \, W^{t+\varepsilon, 2}_{\la_2}(\mathbb R))}.
        \end{eqnarray}
    \end{itemize}
\end{prop}

\medskip

To prove Proposition~\ref{prop4.1}, we use the following lemma.
For $\rho_1>0$ and $\rho_2>0$, we define the double diagonal
set~$\mathcal{D}_{(\rho_1,\rho_2)}$ by
\[
    \mathcal{D}_{(\rho_1,\rho_2)}
    := \left\{(x_1,x_2; y_1,y_2)\in ({X_1\times X_2})\times ({X_1\times X_2})
    : d_1(x_1,y_1)\leq \rho_1~\text{and}~d_2(x_2,y_2)
    \leq \rho_2\right\}.
\]

\begin{lemma}\label{le4.4}
    Suppose that the metric measure spaces $X_i$, $i = 1$, 2,
    satisfy the doubling condition~\eqref{doubling} with
    exponent~$n_i$, and that $L_1$ and $L_2$ are nonnegative
    self-adjoint operators satisfying the finite propagation
    speed property~${\rm (FS)}$. Suppose that $F\in {\mathscr
    S}_{\rm even}$.
    \begin{itemize}
        \item[(i)] If   $\support \widehat{F} \subset [-\rho_1,\,
            \rho_1]\times \mathbb{R}$, then
        $$
        \support K_{F(\SLi,\SLii)}(x_1,x_2,y_1,y_2) \subset \mathcal{D}_{(\rho_1,\infty)}.
        $$

        \item[(ii)] If   $\support \widehat{F} \subset
            \mathbb{R}\times [-\rho_2,\, \rho_2]$, then
        $$
        \support  K_{F(\SLi,\SLii)}(x_1,x_2,y_1,y_2) \subset \mathcal{D}_{(\infty,\rho_2)}.
        $$

        \item[(iii)] If   $\support \widehat{F} \subset
            \big([-\rho_1,\, \rho_1]\times \mathbb{R}\big) \cup
            \big(\mathbb{R}\times [-\rho_2,\, \rho_2]\big)$,
            then
        $$
            \support K_{F(\SLi,\SLii)}(x_1,x_2,y_1,y_2)
            \subset \mathcal{D}_{(\rho_1,\infty)} \cup \mathcal{D}_{(\infty,\rho_2)}.
        $$
    \end{itemize}
\end{lemma}
\begin{proof} Since $F\in {\mathscr S}_{\rm even}$, it follows from  the Fourier
inversion formula that
$$
    F\big(\sqrt{L_1},\sqrt{L_2}\big)
    = \frac{1}{4\pi^2}\int_{\RR^2} \widehat{F}(t_1,t_2)
    \cos(t_1\sqrt{L_1})\cos(t_2\sqrt{L_2}) \;dt_1dt_2,
$$
which gives
\begin{eqnarray*}
\lefteqn{K_{F\big(\sqrt{L_1},\sqrt{L_2}\big)}(x_1,x_2,y_1,y_2)}\\
&& =\frac{1}{4\pi^2}\int_{\RR^2}
  \widehat{F}(t_1,t_2) K_{\cos(t_1\sqrt{L_1})}(x_1,y_1)K_{\cos(t_2\sqrt{L_2})}(x_2,y_2) \;dt_1dt_2.
\end{eqnarray*}
By a standard argument, as in \cite[Lemma 2.1]{COSY},
Lemma~\ref{le4.4} follows.
\end{proof}

\medskip

\begin{proof}[Proof of Proposition~\ref{prop4.1}]

To prove item (i), we consider four cases separately.

\smallskip

\noindent {\bf Case (1):}  $r_1R_1\geq 1\mbox{ and } r_2R_2\geq
1.$

Recall that $\phi$ is a $C_0^\infty(\mathbb{R})$ function such
that $\support \phi \subseteq (1,4)$ and
$\sum_{\ell\in\mathbb{Z}} \phi(2^{-\ell}\lambda) = 1$ for all
$\lambda > 0$.
Set
\[
    \phi_0(\lambda)
    := 1-\sum_{\ell> 0} \phi(2^{-\ell} \lambda),
    \quad\text{and}\quad
    \phi_\ell(\lambda)
    := \phi(2^{-\ell}\lambda) \quad\text{for $\ell \geq 1$.}
\]
Define the functions
\begin{align*}
    \psi(\xi_1,\xi_2)
    &:= 1-\left(1-\phi_0\left({\xi_1\over 2^{j-3}r_1}\right) \right)
        \left(1-\phi_0\left({\xi_2\over 2^{k-3}r_2}\right) \right)\qquad\text{and} \\
    \psi_0(\xi_1,\xi_2)
    &:= \delta_{(1/R_1,1/R_2)}\psi(\xi_1, \xi_2).
\end{align*}
Set ${\widehat {T_\psi F}} :=\psi {\widehat   F} $, where
${\widehat   F}$ denotes the Fourier transform of $F$.
It follows from Lemma~\ref{le4.4}(iii) that
$$
 \support K_{T_\psi F(\SLi,\SLii)}\subseteq \left\{(z_1,z_2,y_1,y_2):
  d_1(z_1,y_1)\leq 2^{j-1}r_1  \ {\rm or}\
  d_2(z_2,y_2)\leq 2^{k-1}r_2   \right\}.
$$
Therefore
\begin{eqnarray}\label{e4.5}
K_{F(\SLi,\SLii)}(z_1,z_2,y_1,y_2)=K_{[F-T_\psi F](\SLi,\SLii)}(z_1,z_2,y_1,y_2),
\end{eqnarray}
when    $d_1(z_1,y_1)>2^{j-1}r_1$ and
$d_2(z_2,y_2)>2^{k-1}r_2$. We may write
\begin{eqnarray}\label{e4.6}
1  &\equiv&\left(\sum_{\ell\geq 0}(\delta_{1/R_1} \phi_\ell) (\lambda_1) \right)
\left(\sum_{ m\geq 0} (\delta_{1/R_2} \phi_m) (\lambda_2)\right).
\end{eqnarray}
Set $\Phi_{\ell, m}(\lambda_1, \lambda_2) :=\phi_\ell (\lambda_1)\phi_m (\lambda_2)$.
Observe that for every $\ell, m\geq 0, $
\begin{eqnarray*}
 \support \left( (\delta_{(1/R_1, 1/R_2)} \Phi_{\ell, m} ) [F-T_\psi
F]\right) \subseteq [-2^{\ell +2}R_1, 2^{\ell +2}R_1]\times [- 2^{m +2}R_2, 2^{m +2}R_2],
 \end{eqnarray*}
   This,    in combination with  equalities \eqref{e4.5} and  \eqref{e4.6} and
   (i) of Lemma~\ref{le4.1}, shows that
\begin{eqnarray}\label{e4.7}
&& \big\|P_{B(x_1,2^jr_1)^c\times B(x_,2^kr_2)^c}F\big(\sqrt{L_1},\sqrt{L_2}\big)P_{B_1(x_1, r_1)\times  B_2(x_2, r_2)}
 \big\|_{L^{p_1}(X_1; \,L^{p_2}(X_2))\to L^2(X_1\times X_2)}\nonumber\\
&&\hskip.5cm \leq   \big\| (F-T_\psi
F)\big(\sqrt{L_1},\sqrt{L_2}\big)P_{B_1(x_1, r_1)\times  B_2(x_2, r_2)}
\big\|_{L^{p_1}(X_1; \, L^{p_2}(X_2))\to L^2(X_1\times X_2)}\nonumber\\
&&\hskip.5cm \leq   \sum_{\ell\geq 0} \sum_{ m\geq 0} \big\| \big((\delta_{(1/R_1, 1/R_2)} \Phi_{\ell, m} )
 [F-T_\psi
F]\big)
 \big(\sqrt{L_1},\sqrt{L_2}\big)P_{B_1(x_1, r_1)\times  B_2(x_2, r_2)}
 \big\|_{L^{p_1}(X_1; \,L^{p_2}(X_2))\to L^2(X_1\times X_2)}\nonumber\\
&&\hskip.5cm \leq   C\prod_{i=1}^2\left({{(R_ir_i)^{n_i} \over
  { V_i(x_i,r_i) }}}\right)^{1/p_i-1/2} \times  M_{\phi, F, R_1, R_2},
 \end{eqnarray}
 where
 \begin{eqnarray}\label{e4.8}
 M_{\phi, F, R_1, R_2} &:=&  \left\| \phi_0 (\lambda_1)\phi_0(\lambda_2)\delta_{(R_1,R_2)}[F-T_\psi
F](\lambda_1, \lambda_2)\right\|_{L^2({\mathbb R^2} )}\nonumber\\
 &&+ \sum_{m\geq 1} 2^{m (n_2(1/p_2-1/2)-1/2)}\left\|  \phi_0 (\lambda_1)\phi_m(\lambda_2)\delta_{(R_1,R_2)}[F-T_\psi
F](\lambda_1, \lambda_2)\right\|_{L^2({\mathbb R^2} )}\nonumber\\
 &&+ \sum_{\ell\geq 1} 2^{\ell (n_1(1/p_1-1/2)-1/2)}\left\|
 \phi_\ell (\lambda_1)\phi_0(\lambda_2)\delta_{(R_1,R_2)}[F-T_\psi
F](\lambda_1, \lambda_2)\right\|_{L^2({\mathbb R^2} )}\nonumber\\
 &&+ \sum_{\ell\geq 1} \sum_{m\geq 1} 2^{m (n_2(1/p_2-1/2)-1/2)} 2^{\ell (n_1(1/p_1-1/2)-1/2)}
\left\|  \phi_\ell (\lambda_1)\phi_m(\lambda_2)\delta_{(R_1,R_2)}[F-T_\psi
F](\lambda_1, \lambda_2)\right\|_{L^2({\mathbb R^2} )}\nonumber\\
 &=:& {\rm (I) + (II) +(III) +(IV)}.
 \end{eqnarray}

For the term ${\rm (I)}$, we recall that $\psi_0(\xi_1,\xi_2) =
\delta_{(1/R_1, 1/R_2)}\psi(\xi_1, \xi_2)$ and use the Plancherel
theorem to obtain
\begin{eqnarray*}
{\rm (I)}\leq C\left\|  \delta_{(R_1,R_2)}[F-T_\psi
F]\right\|_{L^2({\mathbb R^2} )}\leq
C\left\|\left(1-\psi_0\right){\widehat{\delta_{(R_1,R_2)}F}}
\right\|_{L^2({\mathbb R^2} )}.
\end{eqnarray*}
Now we estimate
$\left\|\left(1-\psi_0\right){\widehat{\delta_{(R_1,R_2)}F}}
\right\|_{L^2({\mathbb R^2} )}.$ Observe that if either $\ell <
\ell_0 := \log_2(2^{j-4}r_1R_1)$ or $m < m_0 :=
\log_2(2^{k-4}r_2R_2)$ holds, then
$$
    \phi_\ell  (\xi_1) \phi_m (\xi_2) [1-\psi_0](\xi_1, \xi_2)
    = 0,
$$
which, together with equality~\eqref{e4.6}, gives
\begin{eqnarray}\label{e4.9}\hspace{-1cm}
&& \left\|\left(1-\psi_0\right){\widehat{\delta_{(R_1,R_2)}F}} \right\|_{L^2({\mathbb R^2} )} \\
&&\hskip.5cm= \left\|\left(\sum_{\ell\geq 0}  \phi_\ell (\xi_1) \right)
\left(\sum_{ m\geq 0}   \phi_m (\xi_2)\right)
\left[(1-\psi_0){\widehat{\delta_{(R_1,R_2)}F}}\right](\xi_1, \xi_2) \right\|_{L^2({\mathbb R^2} )}\nonumber\\
&&\hskip.5cm \leq \left\|\left(\sum_{\ell\geq \ell_0} \phi_\ell (\xi_1) \right)
\left(\sum_{ m\geq m_0}   \phi_m (\xi_2)\right)
 {\widehat{\delta_{(R_1,R_2)}F}}(\xi_1, \xi_2) \right\|_{L^2({\mathbb R^2} )}\nonumber\\
&&\hskip.5cm  \leq  2^{-(\ell_0 s+m_0 t)}
   \left\|     \sum_{ m\geq m_0} 2^{m t} \left\|
  \phi_m (\xi_2) \left(\sum_{\ell\geq \ell_0} 2^{s\ell_0} \phi_\ell (\xi_1)
 {\widehat{\delta_{(R_1,R_2)}F}}(\xi_1, \xi_2)\right)
 \right\|_{L^2_{\xi_2}(\mathbb R)}   \right\|_{L^2_{\xi_1}({\mathbb R} )}.\nonumber
 \end{eqnarray}
Let us now estimate the right hand side of \eqref{e4.9}.
It follows from $\support \phi_m  \subset [2^m,2^{m+2}]$ that
for every $\varepsilon > 0$,
\begin{eqnarray}\label{ff}
&&\sum_{ m\geq m_0} 2^{m t} \left\|
  \phi_m (\xi_2) \left(\sum_{\ell\geq \ell_0} 2^{s\ell_0} \phi_\ell (\xi_1)
 {\widehat{\delta_{(R_1,R_2)}F}}(\xi_1, \xi_2)\right)
 \right\|_{L^2_{\xi_2}(\mathbb R)}  \nonumber\\
&&\hskip.5cm = \sum_{ m\geq m_0} 2^{m t}
\left\|\left(1+\xi_2^2\right)^{-(t+\varepsilon)/2}\left(1+\xi_2^2\right)^{(t+\varepsilon)/2}
  \phi_m (\xi_2) \left(\sum_{\ell\geq \ell_0} 2^{s\ell_0} \phi_\ell (\xi_1)
 {\widehat{\delta_{(R_1,R_2)}F}}(\xi_1, \xi_2)\right)
 \right\|_{L^2_{\xi_2}(\mathbb R)}  \nonumber\\
&&\hskip.5cm \leq C\sum_{ m\geq m_0} 2^{m t}2^{-(t+\varepsilon)m}
\left\|\left(1+\xi_2^2\right)^{(t+\varepsilon)/2}
  \phi_m (\xi_2) \left(\sum_{\ell\geq \ell_0} 2^{s\ell_0} \phi_\ell (\xi_1)
 {\widehat{\delta_{(R_1,R_2)}F}}(\xi_1, \xi_2)\right)
 \right\|_{L^2_{\xi_2}(\mathbb R)}  \nonumber\\
&&\hskip.5cm   \leq   C_{\varepsilon}
  \left\| \left(1+\xi_2^2\right)^{(t+\varepsilon)/2}
  \left(\sum_{\ell\geq \ell_0} 2^{s\ell_0} \phi_\ell (\xi_1){\widehat{\delta_{(R_1,R_2)}F}}  (\xi_1, \xi_2)\right)
  \right\|_{L^2_{\xi_2}({\mathbb R})}.
 \end{eqnarray}
This yields
\begin{eqnarray}\label{e4.99}\hspace{-1cm}
 {\rm RHS\ of \ }  \eqref{e4.9}
 &\leq&   C_{\varepsilon}2^{-(\ell_0 s+m_0 t)}
   \left\|
  \left\| \left(1+\xi_2^2\right)^{(t+\varepsilon)/2}
  \left(\sum_{\ell\geq \ell_0} 2^{s\ell_0} \phi_\ell (\xi_1){\widehat{\delta_{(R_1,R_2)}F}}\right)  (\xi_1, \xi_2)
  \right\|_{L^2_{\xi_2}({\mathbb R})}\right\|_{L^2_{\xi_1}({\mathbb R} )} \nonumber\\
    &\leq&  C_{\varepsilon}2^{-(\ell_0 s+m_0 t)}
   \left\|
  \left(1+\xi_2^2\right)^{(t+\varepsilon)/2}
  \sum_{\ell\geq \ell_0} 2^{s\ell} \left\|  \phi_\ell (\xi_1){\widehat{\delta_{(R_1,R_2)}F}} (\xi_1, \xi_2)
  \right\|_{L^2_{\xi_1}({\mathbb R})}\right\|_{L^2_{\xi_2}({\mathbb R} )}.
\end{eqnarray}
On the other hand, we apply a similar argument to that
in~\eqref{ff} to obtain
 \begin{eqnarray*}
    \sum_{\ell\geq \ell_0} 2^{s\ell} \left\|  \phi_\ell (\xi_1){\widehat{\delta_{(R_1,R_2)}F}} (\xi_1, \xi_2)
  \right\|_{L^2_{\xi_2}({\mathbb R})}  &\leq&
      C_{\varepsilon}
   \left\|
  \left(1+\xi_1^2\right)^{(s+\varepsilon)/2}
  {\widehat{\delta_{(R_1,R_2)}F}}(\xi_1, \xi_2)
 \right\|_{L^2_{\xi_1} ({\mathbb R} )},
 \end{eqnarray*}
 which, in combination with estimates \eqref{e4.99}, \eqref{e4.9} and the fact that $\ell_0=\log_2(2^{j-4}r_1R_1)$ and
 $m_0=\log_2(2^{k-4}r_2R_2)$, yields
 \begin{eqnarray*}
  \left\|\left(1-\psi_0\right){\widehat{\delta_{(R_1,R_2)}F}} \right\|_{L^2({\mathbb R^2} )}
  &\leq&  C_{\varepsilon}(2^jr_1R_1)^{-s}(2^kr_2R_2)^{-t}\big\|\delta_{(R_1,R_2)}F
 \big\|_{W^{(s+\varepsilon, \, t+\varepsilon), 2}({\mathbb R}\times {\mathbb R})}
\end{eqnarray*}
as desired.

Now we consider the term ${\rm (II)}$. Set
$$
  {\widetilde\psi}(\xi_1, \xi_2) :=\left(1-\phi_0\left(\xi_1 \over 2^{j-3}r_1R_1 \right)\right)
\phi_0\left(\xi_2 \over 2^{k-3}r_2R_2 \right).
$$
First, we observe that for each $m\geq 1$ and for each fixed
$\lambda_1$, $\delta_{(R_1,R_2)}F(\lambda_1, \lambda_2)$ equals
zero when $\lambda_2$ is in  the support of the
function~$\phi_m $. Thus for $ m\geq 1$ we have
\begin{eqnarray}\label{e4.11}\hspace{1.5cm}
    \left\|\phi_0(\lambda_1)\phi_m(\lambda_2)\delta_{(R_1, R_2)}[F-T_\psi
    F]\right\|_{L^2({\mathbb R}^2)}
    &=&\left\|\phi_0(\lambda_1)\phi_m(\lambda_2)T_{{\widetilde\psi}}  \delta_{(R_1, R_2)}F
     \right\|_{L^2({\mathbb R^2})}\nonumber\\
     &=&\left\|\phi_m(\lambda_2)T_{\phi_0\left(\xi_2 \over 2^{k-3}r_2R_2 \right)}
    \widetilde{F}(\lambda_1, \lambda_2)
     \right\|_{L^2({\mathbb R^2})},
\end{eqnarray}
where $\widetilde{F}$ denotes the function
$$
\widetilde{F}(\la_1,\la_2)
:= \phi_0(\la_1)T_{\left(1-\phi_0\left(\xi_1 \over 2^{j-3}r_1R_1 \right)\right)}\delta_{(R_1,R_2)}F(\la_1,\la_2).
$$
Next, note that $\phi_0\in C_0^{\infty}(\mathbb R)$. Thus for
every $N>0$, there exists a constant $C_N$ such that
$\widehat{\phi_0}(\xi)\leq C_N(1+|\xi|)^{-N}$. This, together
with the fact that $\support \phi_m\subset [2^m, 2^{m+2}]$ and
$\support \widetilde{F}\subset \RR\times [-1,1]$, shows that
for all~$\lambda_1$, $\lambda_2$ we have
\begin{eqnarray}\label{e4.12}
    \Big|\phi_m(\lambda_2)T_{\phi_0\left(\xi_2 \over 2^{k-3}r_2R_2 \right)}
    \widetilde{F}(\lambda_1, \lambda_2)\Big|
    &&\leq C_N\phi_m(\lambda_2) \int_\RR
    \big|\widetilde{F}(\lambda_1, w)\big|
    \frac{2^{k-3}r_2R_2}{(1+2^{k-3}r_2R_2|\lambda_2-w|)^N}\,dw\nonumber\\
    &&\leq  C_N(2^{k}r_2R_22^m)^{-N+1} \phi_m(\lambda_2)
        \|\widetilde{F}(\lambda_1, \cdot) \|_{L^2(\mathbb R)}.
\end{eqnarray}
Again, we can apply a similar argument to that for the
term~${\rm (I)}$ to show that for every~$\lambda_2$,
\[
    \left\|\phi_0(\la_1)T_{\left(1-\phi_0\left(\xi_1 \over 2^{j-3}r_1R_1
        \right)\right)}\delta_{(R_1,R_2)}F(\lambda_1,\lambda_2)
        \right\|_{L_{\la_1}^2(\mathbb R)}
    \leq C(2^jr_1R_1)^{-s}
        \big\| \delta_{(R_1,R_2)}F(\la_1, \lambda_2)
        \big\|_{W^{s+\varepsilon,2}_{\lambda_1}({\mathbb R})}.
\]
This, in combination with \eqref{e4.12}, yields that
\begin{eqnarray*}
    {\rm RHS\ of\ \eqref{e4.11}}
    \leq C_N(2^jr_1R_1)^{-s} (2^{k}r_2R_22^m)^{-N+2}
        \big\| \delta_{(R_1,R_2)}F \big\|_{L^2_{\lambda_2}({\mathbb R}; W^{s+\varepsilon,
        2}_{\lambda_1}({\mathbb R}))}.
\end{eqnarray*}
Choosing $N$ large enough that $N>\max\{n_2/2,t+2\}$, we obtain
\begin{eqnarray*}
    {\rm (II)}
    &\leq& C\sum_{m\geq 1} 2^{m (n_2-1)/2}C_N(2^jr_1R_1)^{-s}
        (2^{k}r_2R_22^m)^{-N+1}
        \big\| \delta_{(R_1,R_2)}F \big\|_{L^2_{\lambda_2}({\mathbb R};
        W^{s+\varepsilon, 2}_{\lambda_1}({\mathbb R}))} \\
    &\leq& C_N (2^jr_1R_1)^{-s}(2^kr_2R_2)^{-t}
        \big\| \delta_{(R_1,R_2)}F \big\|_{L^2_{\lambda_2}({\mathbb R};
        W^{s+\varepsilon,2}_{\lambda_1}({\mathbb R}))}.
\end{eqnarray*}

A symmetric argument yields
\[
    {\rm (III)}
    \leq C (2^jr_1R_1)^{-s}(2^kr_2R_2)^{-t}
        \big\| \delta_{(R_1,R_2)}F \big\|_{L^2_{\lambda_1}({\mathbb R}; \
        W^{t+\varepsilon, 2}_{\lambda_2}({\mathbb R}))}.
\]

Finally, we estimate the term ${\rm (IV)}$. Set
$$
    {\widetilde{\widetilde\psi}}(\xi_1, \xi_2)
    := \phi_0\left(\xi_1 \over 2^{j-3}r_1R_1 \right)
        \phi_0\left(\xi_2 \over 2^{k-3}r_2R_2 \right).
$$
Then for $\ell \geq 1$ and $m \geq 1$, we have
$$
\left\|\phi_\ell(\lambda_1)\phi_m(\lambda_2)\delta_{(R_1, R_2)}[F-T_\psi
F]\right\|_{L^2({\mathbb R^2})}
=\left\|\phi_\ell(\lambda_1)\phi_m(\lambda_2)T_{{\widetilde{\widetilde\psi}}} \, \delta_{(R_1, R_2)}F
 \right\|_{L^2({\mathbb R^2})}.
$$
An argument as in \eqref{e4.12} shows that
\begin{eqnarray}\label{e4.13}
    \lefteqn{\left\|\phi_\ell(\lambda_1)\phi_m(\lambda_2)
        T_{{\widetilde{\widetilde\psi}}} \, \delta_{(R_1, R_2)}F
        \right\|_{L^2({\mathbb R^2})}} \hspace{1cm} \nonumber\\
     &\leq& C_{N_1, N_2}(2^{j}r_1R_12^\ell)^{-N_1+2} (2^{k}r_2R_22^m)^{-N_2+2} \
        \big\| \delta_{(R_1,R_2)}F \big\|_{L^2(\RR^2)}.
\end{eqnarray}
Taking $N_1$ and $N_2$ large enough such that
$N_1>\max\{n_1/2,s+2\}$ and $N_2>\max\{n_2/2,t+2\}$, we obtain
\begin{eqnarray*}
    {\rm (IV)}
    &\leq& C_{N_1, N_2}\sum_{\ell\geq 1} \sum_{m\geq 1}
        2^{\ell (n_1-1)/2} 2^{m (n_2-1)/2}
        (2^{j}r_1R_12^\ell)^{-N_1+2} (2^{k}r_2R_22^m)^{-N_2+2}
        \big\| \delta_{(R_1,R_2)}F \big\|_{L^2(\RR^2)} \\
    &\leq& C (2^jr_1R_1)^{-s}(2^kr_2R_2)^{-t}
        \big\| \delta_{(R_1,R_2)}F \big\|_{W^{(s+\varepsilon,
        t+\varepsilon), 2} ({\mathbb R}\times {\mathbb R})}.
\end{eqnarray*}

Substituting our estimates of ${\rm (I), (II), (III)}$ and
${\rm (IV)}$ back into~\eqref{e4.8}, we see that
\begin{eqnarray*}
    {\rm LHS\ of \ \eqref{e4.7}}
    \leq C\prod_{i=1}^2\left({\big( R_ir_i \big)^{n_i} \over V_i(x_i, r_i)  }
        \right)^{1/p_i-1/2} (2^jr_1R_1)^{-s}(2^kr_2R_2)^{-t}
        \big\| \delta_{(R_1,R_2)}F \big\|_{W^{(s+\varepsilon,  t+\varepsilon), 2}
        ({\mathbb R}\times {\mathbb R})},
\end{eqnarray*}
which implies~\eqref{e4.2} in the case $r_1R_1\geq 1$ and
$r_2R_2\geq 1$.

\medskip

\noindent {\bf Case (2):}  $r_1R_1< 1 $ and $ r_2R_2\geq 1.$

In this case, we have that $r_1<1/R_1$. Let $M_{\phi, F, R_1, R_2}$ be as in \eqref{e4.8}.
By (i) of Lemma~\ref{le4.1}, it can be verified that
\begin{eqnarray}\label{e4.14}
&& \big\|P_{B_1(x_1,2^jr_1)^c\times B_2(x_2,2^kr_2)^c}
 F\big(\sqrt{L_1},\sqrt{L_2}\big)P_{B_1(x_1, r_1)\times
 B_2(x_2, r_2)} \big\|_{L^{p_1}(X_1; \, L^{p_2}(X_2))\to L^2(X_1\times X_2)}\nonumber\\
&&\hskip.5cm \leq   \big\| (F-T_\psi
F)\big(\sqrt{L_1},\sqrt{L_2}\big)P_{B_1(x_1, r_1)\times
B_2(x_2, r_2)} \big\|_{L^{p_1}(X_1; \, L^{p_2}(X_2))\to L^2(X_1\times X_2)}\nonumber\\
&&\hskip.5cm\leq   \big\| (F-T_\psi
F)\big(\sqrt{L_1},\sqrt{L_2}\big)P_{B_1(x_1, 1/R_1)\times
B_2(x_2, r_2)} \big\|_{L^{p_1}(X_1; \, L^{p_2}(X_2))\to L^2(X_1\times X_2)}\nonumber\\
&&\hskip.5cm \leq   C\left({{ 1\over  {  V_1(x_1, R_1^{-1})}}}\right)^{1/p_1-1/2}
 \left({{ (R_2r_2)^{n_2}\over  {  V_2(x_2,r_2) }}}\right)^{1/p_2-1/2} \times  M_{\phi, F, R_1, R_2}.
 \end{eqnarray}
Substituting our estimates of ${\rm (I), (II), (III)}$ and
${\rm (IV)}$ of {\bf Case (1)} back into~\eqref{e4.14}, we find
that
\begin{eqnarray*}
    {\rm LHS\ of \ \eqref{e4.14}}
    \leq  C\left({{ 1\over  {  V_1(x_1, R_1^{-1})  }}}\right)^{1/p_1-1/2}
    \left({{ (R_2r_2)^{n_2}\over  {  V_2(x_2,r_2) }}}\right)^{1/p_2-1/2}
     (2^jr_1R_1)^{-s}(2^kr_2R_2)^{-t}   \big\|
    \delta_{(R_1,R_2)}F \big\|_{W^{(s+\varepsilon, \,  t+\varepsilon), 2}
    ({\mathbb R}\times {\mathbb R})},
\end{eqnarray*}
which implies \eqref{e4.2} in the case $r_1R_1 < 1 $ and
$r_2R_2 \geq 1.$

\medskip

\noindent {\bf Case (3):}  $r_1R_1\geq  1 $ and $ r_2R_2< 1.$

By a symmetric argument to that in {\bf Case (2)}, we can prove
{\bf Case (3)}.


\medskip

\noindent {\bf Case (4):}  $r_1R_1< 1 $ and $ r_2R_2< 1.$

A similar argument to that for {\bf Case (2)} establishes {\bf
Case (4)}.

This completes the proof of (i). The proofs of items (ii) and
(iii) can be obtained by making minor modifications to that of
item (i); we omit the details. This completes the proof of
Proposition~\ref{prop4.1}.
\end{proof}

We end this section by stating Proposition~\ref{prop4.2}, which
is the analogue for restriction type estimates of the
form~$({\rm{ {ST}}}_{p_i,2}^{\infty})$ of
Proposition~\ref{prop4.1} for~$({\rm{ {ST}}}_{p_i,2}^{2})$. The
proof can be obtained by making minor modifications to that of
Proposition~\ref{prop4.1}; we omit the details.

\begin{prop}\label{prop4.2}
    Suppose that the metric measure spaces $X_i$, $i = 1$, 2,
    satisfy the doubling condition~\eqref{doubling} with
    exponent~$n_i$. Further assume that $L_i$, $i = 1$, 2, are
    nonnegative self-adjoint operators satisfying the finite
    propagation speed property~${\rm (FS)}$ and restriction type
    estimates $({\rm{ {ST}}}_{p_i,2}^{\infty})$, for some $p_i$
    with $1 \leq p_i < 2$.  Suppose that $F\in {\mathscr S}_{\rm
    even}$.
    \begin{itemize}
        \item[(i)]   If the function $F$ is supported on
            $[-R_1,R_1]\times [-R_2,R_2]$, then for each $s>0$
            and $t>0$, for each $\varepsilon > 0$ there exists
            a constant $C_{s,t}$ such that for every
            $B_1(x_1,r_1)\times B_2(x_2,r_2)$ and for every
            $j$, $k = 1,2,\ldots,$
        \begin{eqnarray*}
         \lefteqn{\left\|P_{B(x_1,2^jr_1)^c\times B(x_,2^kr_2)^c}
        F\big(\sqrt{L_1},\sqrt{L_2}\big)P_{B_1(x_1, r_1)\times  B_2(x_2, r_2)}
         \right\|_{L^{p_1}(X_1; \, L^{p_2}( X_2))\to L^2(X_1\times X_2) }}\nonumber\hspace{-0.3cm}\\
         &\leq& C_{s,t}\prod_{i=1}^2
         \max\left\{
         {\big( R_ir_i \big)^{n_i} \over V_i(x_i, r_i)},
          {1\over {V_i(x_i,R_i^{-1})}}
          \right\}^{1/p_i-1/2}
         (2^jr_1 R_1)^{-s} (2^kr_2 R_2)^{-t}
        \big\|\delta_{(R_1,R_2)}F\big\|_{W^{(s+\varepsilon,\, t+\varepsilon), \infty} ({\mathbb R\times \mathbb R})};
        \hspace{-0.3cm}
        \end{eqnarray*}

        \item[(ii)] If the function $F$ is supported on
            $[-R_1,R_1]\times \RR$, then for each $s>0$ and for
            each $\varepsilon > 0$, there exists a constant
            $C_s$ such that for every ball  $B_1(x_1,r_1)$, and
            for every $j = 1,2,\ldots,$
        \begin{eqnarray*}
        \lefteqn{\left\|P_{B(x_1,2^jr_1)^c\times X_2}F\big(\sqrt{L_1},\sqrt{L_2}\big)P_{B_1(x_1, r_1)\times X_2}
        \right\|_{L^{p_1}(X_1; \, L^2(X_2))\to L^2(X_1\times X_2)}} \nonumber\\
         &\leq& C_s \max\left\{{\big( R_1r_1 \big)^{n_1} \over V_1(x_1, r_1)  }, {1\over {V_1(x_1,R_1^{-1})}}
         \right\}^{1/p_1-1/2}  (2^jr_1R_1)^{-s}
         \big\|\delta_{(R_1,1)}F\big\|_{{L^{\infty}_{\la_2}({\mathbb R}; \, W^{s+\varepsilon, \infty}_{\la_1}(\mathbb R))}};
        \end{eqnarray*}

        \item[(iii)]    If the function $F$ is supported on
            $\RR\times [-R_2,R_2]$, then for each $t>0$ and for
            each $\varepsilon>0$ there exists a constant $C_t$
            such that for every  ball $B_2(x_2,r_2)$ and for
            every $k = 1,2,\ldots,$
        \begin{eqnarray*}
         \lefteqn{\left\| P_{X_1\times B_2(x_2,2^kr_2)^c}F\big(\sqrt{L_1},\sqrt{L_2}\big)P_{X_1\times B_2(x_2, r_2)}
        \right\|_{L^{p_2}(X_2; \, L^2(X_1))\to L^2(X_1\times X_2)}} \nonumber\\
         &\leq& C_t \max\left\{{\big( R_2r_2 \big)^{n_2} \over V_2(x_2, r_2)  }, {1\over {V_2(x_2,R_2^{-1})}}
         \right\}^{1/p_2-1/2}  (2^kr_2
        R_2)^{-t}\big\|\delta_{(1,R_2)}F\big\|_{L^{\infty}_{\la_1}({\mathbb
        R}; \, W^{t+\varepsilon,\infty}_{\la_2}(\mathbb R))}.
        \end{eqnarray*}
    \end{itemize}
\end{prop}

\medskip

\section{Proofs of Theorems~\ref{th3.1} and~\ref{th3.2} }
\label{sec:proofofmainthms}
\setcounter{equation}{0}

The proof of Theorems~\ref{th3.1} and~\ref{th3.2}  is based on
a pair of lemmata. The following lemma is a standard result in
the theory of singular integrals on product spaces. It is a
version of \cite[Lemma 3.5]{DLY} adjusted to the setting of
spaces of homogeneous type.

\begin{lemma}\label{le5.1}
Fix  $N>\max\{n_1, n_2\}/4$. Assume that $T$ is a linear
operator, or a nonnegative sublinear operator, satisfying the
weak-type (2,2) bound
\begin{eqnarray*}
\left|\ \{x\in  {X_1\times X_2} :
|Tf(x)|>\eta\}\ \right|\leq
C_T\eta^{-2}\|f\|_{L^2( {X_1\times X_2} )}^2, \
\ {\text{for all}}~\eta>0,
\end{eqnarray*}
 and that for every $(H^1_{L_1, L_2}, 2, N)$-atom $a$, we have
\begin{eqnarray}\label{e5.1}
\|Ta\|_{L^1( {X_1\times X_2} )}\leq C
\end{eqnarray}
with constant $C$ independent of $a$. Then $T$ is bounded from
$\mathbb{H}^1_{L_1,L_2,at,N}( {X_1\times X_2} )$ to $L^1(
{X_1\times X_2} )$, and
$$
    \|Tf\|_{L^1( {X_1\times X_2} )}
    \leq C\|f\|_{\mathbb{H}^1_{L_1,L_2,at,N}(X_1\times X_2)}.
$$
Consequently, by density, $T$ extends to a bounded operator
from $H^1_{L_1,L_2,at,N}( {X_1\times X_2} )$ to $L^1(
{X_1\times X_2} )$.
\end{lemma}

\begin{proof}
Let $f \in \mathbb{H}^1_{L_1,L_2,at,N}( {X_1\times X_2}) $,
where $f = \sum \lambda_j a_j$ is an atomic $(H^1_{L_1, L_2},
2, N)$-representation such that
$$
    \|f\|_{ \mathbb{H}^1_{L_1,L_2,at,N}( {X_1\times X_2})}
    \sim \sum_{j=0}^\infty |\lambda_j| .
$$
Since the sum converges in $L^2$ (by definition), and since $T$
is of weak type $(2,2)$, we have that at almost every point,
\begin{equation}\label{e5.2}
    |T(f)| \leq \sum_{j=0}^\infty |\lambda_j| \,|T(a_j)|.
\end{equation}
Indeed, for every $\eta >0$, we have that, if $f^K:= \sum_{j>K}
\lambda_j a_j$, then
\begin{eqnarray*}
\big|\ \{x: |Tf(x)| - \sum_{j=0}^\infty |\lambda_j| \,|T a_j(x)| >\eta\}\big|\,&\leq& \limsup_{K\to \infty}
\big|  \{x:  |Tf^K(x)|>\eta\}\big|\\
&\leq& \,C_T\,\,\eta^{-2} \,\limsup_{N\to \infty} \|f^K\|_2^2 =0,
\end{eqnarray*}
from which (\ref{e5.2}) follows.
In turn, (\ref{e5.1}) and (\ref{e5.2}) imply the desired $L^1$ bound for $Tf$.
\end{proof}

%
%
%
%
%
%

Now we recall Journ\'e's covering lemma (see \cite{J1,P}) in
the setting of spaces of homogeneous type. Let
$(X_i,d_i,\mu_i)$, $i = 1$, 2, be spaces of homogeneous type
and let $\{I_{\alpha_i}^{k_i}\subset X_i\}$, $i = 1$, 2, be the
same dyadic cubes as in Theorem~\ref{le5.2}.  The open set
$I_{\alpha_1}^{k_1} \times I_{\alpha_2}^{k_2}$ for $k_1$,
$k_2\in \mathbb{Z}$, $\alpha_1\in \Lambda_{k_1}$ and
$\alpha_2\in \Lambda_{k_2}$ is called a dyadic rectangle of
$X_1\times X_2$. Let $\Omega\subset X_1\times X_2$ be an open
set of finite measure. Denote by $m(\Omega)$ the maximal dyadic
subrectangles in $\Omega$. For $i = 1$, 2, denote by
$m_{i}(\Omega)$ the family of dyadic rectangles
$R\subset\Omega$ which are maximal in the $x_i$-direction. In
what follows, we denote by $R :=I_1\times I_2$ any dyadic
rectangle of $X_1\times X_2$. Given $R = I_1\times I_2\in
m_1(\Omega)$, let $I^{\ast}_2 $ be the largest dyadic cube
containing $I_2$ such that
\[
    V(I_1\times I^{\ast}_2\cap \Omega)
    > {1\over2 }V(I_1\times I^{\ast}_2).
\]
Similarly, given $R = I_1\times I_2\in m_2(\Omega)$, let
$I^{\ast}_1$ be the largest dyadic cube containing $I_1$ such
that
$$
V(I^{\ast}_1\times I_2 \cap \Omega)>{1\over2 }V(I^{\ast}_1\times I_2).
$$
The following lemma is proved in~\cite{HLL}.

\begin{lemma}[\cite{HLL}]\label{le5.3}
    Let all the notation be the same as above. Assume that
    $\Omega\subset X_1\times X_2$ is an open set with finite measure.
    Then for each $\delta > 0$, there exists a constant~$C$
    such that
    \begin{eqnarray}\label{e5.3}
        \sum_{R=I_1\times I_2\in m_1(\Omega)}V(R)
            \left({\ell(I_2)\over\ell(I^{\ast}_2)}\right)^{\delta}
        \leq CV(\Omega)
    \end{eqnarray}
    and
    \begin{eqnarray}\label{e5.4}
        \sum_{R=I_1\times I_2\in m_2(\Omega)}V(R)
            \left({\ell(I_1)\over\ell(I^{\ast}_1)}\right)^{\delta}
        \leq CV(\Omega).
    \end{eqnarray}
\end{lemma}


\medskip

With these results in hand, we are ready to prove
Theorem~\ref{th3.1}.

\begin{proof}[Proof of Theorem~\ref{th3.1}]
Observe  that
$$
    \sup_{t_1, t_2>0}
        \big\| \eta_{(1, 2)}\delta_{(t_1, t_2)}F \big\|_{W^{(s_1, s_2), 2}
        ({\mathbb R}\times {\mathbb R})}
    \sim \sup_{t_1,t_2>0}
        \big\| \eta_{(1,2)}\delta_{(t_1, t_2)}G \big\|_{W^{(s_1,  s_2), 2}
        ({\mathbb R}\times {\mathbb R})},
$$
where $G(\lambda_1,\lambda_2) =
F\big(\sqrt{\lambda_1},\sqrt{\lambda_2})$. Hence we can replace
$F(L_1,L_2)$ by $F\big(\sqrt{L_1},\sqrt{L_2}\big)$ in the
proof.

To prove (i) of Theorem~\ref{th3.1}, by Lemma~\ref{le4.2} it
suffices to show that $F\big(\sqrt{L_1},\sqrt{L_2}\big)$ is
uniformly bounded on each $(H^1_{L_1, L_2}, 2, N)$-atom $a$
with $N > \max\{{n_1/4},{n_2/4}\}$, that is, there exists a
constant $C > 0$ independent of $a$ such that
\begin{eqnarray}\label{e5.5}
    \big\| F\big(\sqrt{L_1},\sqrt{L_2}\big)(a) \big\|_{L^1(X_1\times X_2)}
    \leq C.
\end{eqnarray}

From Definition~\ref{def2.5} of $(H^1_{L_1, L_2}, 2, N)$-atoms,
we have that the atom~$a$ is supported in an open set
$\Omega\subset X_1\times X_2$ with finite measure and $a$ can
be further decomposed into $ a = \sum_{R\in m(\Omega)} a_R$.
For each dyadic rectangle $R = I \times J\in m(\Omega)$, let $
I^* $ be the largest dyadic cube containing $I$ such that $ I^*
\times J\subset\widetilde{\Omega}$, where $\widetilde{\Omega}
:=\{x\in {X_1\times X_2} :\ M_s(\chi_{\Omega})(x)>1/2\}$. Here
$M_s$ denotes the strong maximal operator. Next, let $ J^* $ be
the largest dyadic cube containing $J$, so that $ I^* \times
J^* \subset\widetilde{\widetilde{\Omega}}$, where
$\widetilde{\widetilde{\Omega}}:=\{x\in  {X_1\times X_2} :\
M_s(\chi_{\widetilde{\Omega}})(x)>1/2\}$. Now let
$\widetilde{R}$ be the 100-fold dilate of $ I^* \times  J^* $,
as defined in Section~\ref{sec:Hardyatomic}. An application of
the strong maximal function theorem shows that
$V(\cup_{R\subset\Omega} \widetilde{R})\leq
CV(\widetilde{\widetilde{\Omega}})\leq
CV(\widetilde{\Omega})\leq CV(\Omega)$. From (iii) in the
definition of $(H^1_{L_1, L_2}, 2, N)$-atoms, we obtain that
\begin{eqnarray*}
    \big\| F\big(\sqrt{L_1},\sqrt{L_2}\big)a \big\|_{L^1(\cup\widetilde{R})}
    \leq V(\cup\widetilde{R})^{1/2}
        \big\| F\big(\sqrt{L_1},\sqrt{L_2}\big)a \big\|_{L^2( {X_1\times X_2} )}
    \leq C V(\Omega)^{1/2}\|a\|_{L^2( {X_1\times X_2} )}
    \leq C.
\end{eqnarray*}
Here the union $\cup\widetilde{R}$ is over all dyadic
rectangles $R = I \times J\in m(\Omega)$. Therefore, the proof
of (\ref{e5.5}) reduces to showing that
\begin{eqnarray}\label{e5.6}
    \int_{\big(\cup \widetilde{R}\big)^c}
    |F\big(\sqrt{L_1},\sqrt{L_2}\big)(a)(x)| \, d\mu(x)
    \leq C,
\end{eqnarray}
where $x = (x_1,x_2) \in X_1\times X_2$ and $\mu = \mu_1 \times
\mu_2$.  Since $a=\sum_{R\in m(\Omega)}a_R$, we can write
\begin{eqnarray}\label{e5.7}
    &&\int_{\big(\cup \widetilde{R}\big)^c}
        |F\big(\sqrt{L_1},\sqrt{L_2}\big) a (x)| \, d\mu(x)
    \leq \sum_{R\in m(\Omega) }
    \int_{\widetilde{R\,}^c}|F\big(\sqrt{L_1},\sqrt{L_2}\big) a_R (x)| \, d\mu(x)\nonumber\\
    &&\hskip.5cm \leq \left( \sum_{R\in m(\Omega) } \int_{(100 I^* )^c\times
    X_2}
     + \sum_{R\in m(\Omega) }\int_{X_1\times
    (100 J^* )^c}\right) |F\big(\sqrt{L_1},\sqrt{L_2}\big) a_R(x)| \, d\mu(x)\\
    &&\hskip.5cm =: D+E.\nonumber
\end{eqnarray}

Let us estimate the term $D$. One writes
\begin{eqnarray}\label{e5.8}\hspace{1cm}
    && \hspace{-1cm}\sum_{R\in m(\Omega) }\int_{(100 I^* )^c\times X_2}
    |F\big(\sqrt{L_1},\sqrt{L_2}\big)(a)(x)| \, d\mu(x) \\
     &=&
      \sum_{R\in m(\Omega) }\left(\int_{(100 I^* )^c\times 100J}
     +  \int_{(100 I^* )^c\times (100J)^c}\right)
        |F\big(\sqrt{L_1},\sqrt{L_2}\big)(a)(x)| \, d\mu(x)
     =:  D_1+D_2.\nonumber
\end{eqnarray}

\medskip

\noindent {\bf  {\underline{Estimate of $D_1$}}.}\  Choose
$M\in{\mathbb N}$ large enough that $M > s/2$. Recall that $R =
I\times J \in m(\Omega)$, and $\ell(I)$ denotes the diameter
of~$I$. Following (8.7) and (8.8) in~\cite{HM}, we write
\begin{eqnarray}
&&\hspace{-1.2cm}1\!\!1_1=2 \Big( \ell (I)^{-2}
    \int_{\ell (I)}^{\sqrt{2}\ell (I)}w\,dw\Big)   \cdot 1\!\!1_1\nonumber\\
&=&2\ell (I)^{-2} \int_{\ell (I)}^{\sqrt{2}\ell
(I)}w(1\!\!1_1-e^{-w^2L_1})^M\,dw
 +  \sum_{\kappa=1}^M c_{\kappa,M}  \ell (I)^{-2}
\int_{\ell(I)}^{\sqrt{2}\ell (I)} w e^{-\kappa w^2L_1} \,dw,
\label{e5.9}
\end{eqnarray}
where $c_{\kappa, M}\in \mathbb{R}$ are constants depending
only on $\kappa$ and~$M$. However, $\partial_we^{-\kappa
w^2L_1} = -2\kappa w L_1e^{-\kappa w^2L_1} $, and therefore
\begin{eqnarray}
2\kappa L_1\int_{\ell (I)}^{\sqrt{2}\ell (I)} w e^{-\kappa w^2L_1} \, dw
&=&e^{-\kappa \ell (I)^2L_1}-e^{-2\kappa \ell (I)^2L_1}
 \nonumber\\
&=&  e^{-\kappa \ell (I)^2L_1}
(1\!\!1_1-e^{-\ell (I)^2L_1})    \sum_{i=0}^{\kappa-1}e^{-i\ell (I)^2L_1}.
\label{e5.10}
\end{eqnarray}
Applying the procedure outlined in (\ref{e5.9})
and~(\ref{e5.10}) $M$ times, we have
 \begin{eqnarray}
1\!\!1_1
   &=&  2^M \left(\ell (I)^{-2}\int_{\ell (I)}^{\sqrt{2}\ell (I)}w(1\!\!1_1-e^{-w^2L_1})^M\, dw\right)^{M}    \nonumber\\
  &&+ \sum_{m=1}^{ M} \ell(I)^{-2m} \sum_{\kappa=1}^{(2M-1)m} c_{m, \kappa, M} e^{-\kappa\ell (I)^2L_1}
 \times \nonumber\\
 && \hspace{1.5cm} \times  (1\!\!1_1-e^{-\ell (I)^2L_1})^{m}
  \left(2\ell (I)^{-2}\int_{\ell (I)}^{\sqrt{2}\ell (I)}w(1\!\!1_1-e^{-w^2L_1})^M\, dw\right)^{M-m}
 L_1^{-m}.
 \label{e5.11}
 \end{eqnarray}
Recall that  $a_{R}=(L_1^M\otimes L_2^M ) b_R$. For every
$0\leq m_i\leq M$, $i = 1$, 2, we set
$$
    a_R^{(m_1, m_2)}(x)
    := (L_1^{M-m_1} \otimes L_2^{M-m_2}) b_R(x).
$$
With this notation, we apply~\eqref{e5.11} to obtain
 \begin{eqnarray}\label{e5.12}
 &&\hspace{-1cm}F\big(\sqrt{L_1},\sqrt{L_2}\big)a(x) \nonumber\\
  & =&  \sum_{m=0}^{ M-1}  \ell (I)^{-2m} G_{m,   M}(L_1)
\left(\ell (I)^{-2 }\int_{\ell (I)}^{\sqrt{2}\ell
(I)}w(1\!\!1_1-e^{-w^2L_1})^M \, dw\right)
F\big(\sqrt{L_1},\sqrt{L_2}\big) a_R^{(m, 0)}(x)  \nonumber\\
 &&+  \,  \ell (I)^{-2M}\sum_{\kappa=1}^{(2M-1)M}c_{M, \kappa, M}
 e^{-\kappa\ell (I)^2L_1} (1\!\!1_1-e^{-\ell (I)^2L_1})^{M}F\big(\sqrt{L_1},\sqrt{L_2}\big)a_R^{(M, 0)}(x),
  \end{eqnarray}
where $G_{0,   M}(L_1) := 1\!\!1_1$ and for  $m=1, \ldots,
M-1$,
$$
  G_{m,  M}(L_1) :=\sum_{\kappa=1}^{(2M-1)m}c_{m, \kappa, M} e^{-\kappa\ell (I)^2L_1}(1\!\!1_1 - e^{-\ell (I)^2L_1})^{m}
\left(2\ell (I)^{-2}\int_{\ell (I)}^{\sqrt{2}\ell (I)}w(1\!\!1_1 -
e^{-w^2L_1})^M\, dw\right)^{M-m-1}.
 $$
Substituting \eqref{e5.12} back into the term $D_1$, we   estimate the term $D_1$ by examining  $m$ in two cases:
$m=0,1, \ldots, M-1$ and   $m=M$.

\medskip

\noindent
{\bf Case (1) }\     $m=0,1, \ldots, M-1$. \

\smallskip

In this case, we want to estimate
\begin{eqnarray}\label{eddd}
     \int_{(100 I^* )^c\times 100J}  \left|G_{m,   M}(L_1)
    \left(\ell (I)^{-2}\int_{\ell (I)}^{\sqrt{2}\ell (I)}w(1\!\!1_1 - e^{-w^2L_1})^M
    F\big(\sqrt{L_1},\sqrt{L_2}\big)a^{(m, 0)}_R(x) \, dw\right)\right| \,d\mu(x)\nonumber\\
    \leq \ell (I)^{-2} \int_{\ell (I)}^{\sqrt{2}\ell (I)}
     w\left\|G_{m,   M}(L_1)   (1\!\!1_1 - e^{-w^2L_1})^M
    F\big(\sqrt{L_1},\sqrt{L_2}\big)a^{(m, 0)}_R\right\|_{L^1((100 I^* )^c\times 100J)} \, dw.
\end{eqnarray}
Recall that $\phi$ is a nonnegative $C_0^{\infty}$ function as
in \eqref{e3.1}.
Then for every $w\in [\ell(I), \sqrt{2}\ell (I)]$ and for all
$\lambda_1$, $\l_2 \geq 0$,
\begin{eqnarray*}
G_{m,   M}(\l_1^2) (1-e^{-w^2\l_1^2})^MF(\l_1,\l_2)&=&\sum_{\ell\in \mathbb{Z}}\phi(2^{-\ell}\l_1)
G_{m,   M}(\l_1^2)(1-e^{-w^2\l_1^2})^MF(\l_1,\l_2).
\end{eqnarray*}
Set
$$
F_{w,m, M}^{\ell}(\l_1,\l_2):=\phi(2^{-\ell}\l_1)G_{m,   M}(\l_1^2)(1-e^{-w^2\l_1^2})^MF(\l_1,\l_2).
$$
Therefore,
\begin{eqnarray}\label {e5.13}
&&\left\|G_{m,   M}(L_1)   (1\!\!1_1 - e^{-w^2L_1})^M
F\big(\sqrt{L_1},\sqrt{L_2}\big)a^{(m, 0)}_R\right\|_{L^1((100 I^* )^c\times 100J)}\nonumber\\
&&\hskip.5cm \leq \sum_{\ell\in \mathbb{Z}}
 \left\|F_{w,m, M}^{\ell}\big(\sqrt{L_1},\sqrt{L_2}\big)a^{(m, 0)}_R\right\|_{L^1((100 I^* )^c\times 100J)}\nonumber\\
&&\hskip.5cm \leq \sum_{\ell\in \mathbb{Z}}\sum_{j=6}^{\infty}\left(V(U_j(I^{\ast} )\times 100J) \right)^{1/2}\left\|F_{w,m, M}^{\ell}(L_1, L_2)
 a^{(m, 0)}_R\right\|_{L^2(U_j(I^{\ast} )\times 100J)}  \nonumber\\
&&\hskip.5cm  \leq \sum_{\ell\in \mathbb{Z}}\sum_{j=6}^{\infty}2^{jn_1/2}\left({V_1(I^{\ast})\over V_1(I)}\right)^{1/2} V(R)^{1/2}
 \left\|F_{w,m, M}^{\ell}(L_1, L_2)
a^{(m, 0)}_R\right\|_{L^2(U_j(I^{\ast} )\times 100J)}.
\end{eqnarray}
To go further, there are two cases to consider: $2^\ell
\ell(I)\geq 1$ and $2^\ell \ell(I)\leq 1$.

\medskip

\noindent
{\bf Subcase (1.1)}\ $2^\ell \ell(I)\geq 1$.

\smallskip

In this case, we note that $\support F_{w,m, M}^{\ell}\subseteq
[-2^{\ell + 2},2^{\ell + 2}]\times \RR$ and $\support a^{(m,
0)}_R\subseteq 10R$.  By (ii) of Proposition~\ref{prop4.1},
choosing $n_1/2 < s < 2M$, we have
\begin{eqnarray}\label {e5.14}
    && \big\| F_{w,m, M}^{\ell}(\SLi,\SLii) a^{(m, 0)}_R
        \big\|_{L^2(U_j(I^{\ast} )\times 100J)} \nonumber\\
    &&\hskip.5cm \leq \big\| P_{U_j(I^{\ast})\times X_2}
        F_{w,m, M}^{\ell}(\SLi,\SLii)P_{ 10I\times X_2}
        \big\|_{L^{p_1}(X_1; \, L^2(X_2))\to L^2(X_1\times X_2)}
        \ \big\| a^{(m, 0)}_R \big\|_{L^{p_1}(X_1; \, L^2(X_2))} \nonumber\\
    &&\hskip.5cm \leq C\left({(2^\ell \ell(I))^{n_1} \over V_1(x_I, \ell(I))}\right)^{1/p_1-1/2}
        (2^{j+\ell}\ell(I^{\ast}))^{-s}
        \ \big\| \delta_{(2^\ell, 1)}F_{w,m, M}^{\ell} \big\|_{{L^{\infty}_{\la_2}({\mathbb R};
        \, W^{s,2}_{\la_1}(\mathbb R))}}
        \ \big\| a^{(m, 0)}_R \big\|_{L^{p_1}(X_1; \,L^2(X_2))} \nonumber\\
    &&\hskip.5cm \leq C (2^\ell \ell(I))^{n_1(1/p_1-1/2)}(2^{j+\ell}\ell(I^{\ast}))^{-s}
        \ \big\| \delta_{(2^\ell, 1)}F_{w,m, M}^{\ell} \big\|_{{L^{\infty}_{\la_2}({\mathbb R};
        \, W^{s,2}_{\la_1}(\mathbb R))}}
        \ \big\| a^{(m, 0)}_R \big\|_{L^2(X_1\times X_2)}\nonumber\\
    &&\hskip.5cm \leq C 2^{-js}(2^\ell \ell(I))^{n_1(1/p_1-1/2)-s}
        \left({\ell(I)\over \ell(I^{\ast})}\right)^s
        \ \big\| \delta_{(2^\ell, 1)}F_{w,m, M}^{\ell} \big\|_{{L^{\infty}_{\la_2}
        ({\mathbb R}; \, W^{s,2}_{\la_1}(\mathbb R))}}
        \ \big\| a^{(m, 0)}_R \big\|_{L^2(X_1\times X_2)}.
\end{eqnarray}

\smallskip

\noindent
{\bf Subcase (1.2)} \     $2^\ell\ell(I)<1$.

By (ii) of Proposition~\ref{prop4.1}, choosing $n_1/2 < s <
2M$, we have
\begin{eqnarray}\label {e5.15}
    && \big\|F_{w,m, M}^{\ell}(\SLi,\SLii) a^{(m, 0)}_R
        \big\|_{L^2(U_j(I^{\ast})\times 100J)} \nonumber\\
    &&\hskip.5cm \leq \big\|P_{U_j(I^{\ast})\times X_2}F_{w,m, M}^{\ell}(\SLi,\SLii)P_{10I\times X_2}
        \big\|_{L^{p_1}(X_1; \, L^2(X_2))\to L^2(X_1\times X_2)} \
        \big\|a^{(m, 0)}_R\big\|_{L^{p_1}(X_1; \, L^2(X_2))}\nonumber\\
    &&\hskip.5cm \leq C\left({1 \over V_1(x_I, 2^{-\ell})}\right)^{1/p_1-1/2}(2^{j+\ell}(I^{\ast}))^{-s}
        \big\|\delta_{(2^\ell, 1)}F_{w,m, M}^{\ell}
        \big\|_{{L^{\infty}_{\la_2}({\mathbb R}; \, W^{s,2}_{\la_1}(\mathbb R))}} \
        \big\|a^{(m, 0)}_R\big\|_{L^{p_1}(X_1; \, L^2(X_2))}\nonumber\\
    &&\hskip.5cm \leq C\left({V_1(x_I,\ell(I)) \over V_1(x_I, 2^{-\ell})}\right)^{1/p_1-1/2}
        (2^{j+\ell}\ell(I^{\ast}))^{-s}
        \big\|\delta_{(2^\ell, 1)}F_{w,m, M}^{\ell}
        \big\|_{{L^{\infty}_{\la_2}({\mathbb R}; \, W^{s,2}_{\la_1}(\mathbb R))}} \
        \big\|a^{(m, 0)}_R\big\|_{L^2(X_1\times X_2)}\nonumber\\
    &&\hskip.5cm \leq C 2^{-js}(2^\ell \ell(I))^{-s}\left({\ell(I)\over \ell(I^{\ast})}\right)^s
        \big\|\delta_{(2^\ell, 1)}F_{w,m, M}^{\ell}
        \big\|_{{L^{\infty}_{\la_2}({\mathbb R}; \, W^{s, 2}_{\la_1}(\mathbb R))}} \
        \big\|a^{(m, 0)}_R\big\|_{L^2(X_1\times X_2)}.
\end{eqnarray}
On the other hand, we use the Sobolev embedding theorem to
obtain that for every $u\in [\ell(I), \sqrt{2}\ell (I)]$,
\begin{eqnarray}\label{eq:one-variablecondition}
    &&\big\|\delta_{(2^\ell, 1)}F_{w,m, M}^{\ell}
        \big\|_{{L^{\infty}_{\la_2}({\mathbb R}; \, W^{s, 2}_{\la_1}(\mathbb R))}}\nonumber\\
    &&\hskip.5cm = \sup_{\la_2}
    \left\|\left( I+{\partial^2\over \partial x^2_1}\right)^{s/2}
        \delta_{(2^\ell,1)}F_{w,m, M}^{\ell}(\cdot,\la_2)
        \right\|_{L^{2}_{\la_1}({\mathbb R})}\nonumber\\
    &&\hskip.5cm \leq \left\{\sup_{t_2>0} \big\|\eta_{(1,2)}\delta_{(2^\ell,t_2)}F
        \big\|_{W^{(s,1/2+\varepsilon), 2}({\mathbb R}\times {\mathbb R} )}
        + \big\|\eta_{1}\delta_{(2^\ell, 1)}F(\cdot,0)\big\|_{W^{s,2}({\mathbb R})}\right\}
        \min\left\{1, (2^{\ell}w)^{2M}\right\} \nonumber\\
    &&\hskip.5cm \leq C\min\left\{1, (2^{\ell}\ell(I))^{2M}\right\}.
\end{eqnarray}
Since we chose $M > s/2$, we have
$$
 \sum_{\ell\in {\mathbb Z}: 2^{\ell}\ell(I)\geq 1}  (2^\ell
\ell(I))^{n_1(1/p_1-1/2)-s} + \sum_{\ell\in {\mathbb Z}: 2^{\ell}\ell(I)\leq 1}
 (2^\ell
\ell(I))^{2M-s}  \leq C.
$$
This, in combination with
estimates~\eqref{e5.13}--\eqref{eq:one-variablecondition},
yields that the right-hand side of inequality~\eqref{e5.13} is
bounded by
\begin{eqnarray}\label {e5.166}
    \lefteqn{\sum_{j=1}^{\infty}  2^{-j(s-{n_1\over 2})}
        \left({\ell(I)\over \ell(I^{\ast})}\right)^{s - \frac{n_1}{2}} V(R)^{1/2} \
        \big\|a^{(m, 0)}_R\big\|_{L^2(X_1\times X_2)}} \hspace{1cm} \nonumber\\
    & \leq & C   \left({\ell(I)\over \ell(I^{\ast})}\right)^{s - \frac{n_1}{2}} V(R)^{1/2} \
        \big\|a^{(m, 0)}_R\big\|_{L^2(X_1\times X_2)}.
\end{eqnarray}

\medskip
\noindent
{\bf Case (2) }\     $m=M$. \

\smallskip

An argument similar to that in {\bf Case (1)} shows that
\begin{eqnarray*}
    && \int_{(100 I^* )^c\times 100J}
        \left|\sum_{\kappa=1}^{(2M-1)M} c_{M, \kappa, M}
        e^{-\kappa\ell (I)^2L_1} (1\!\!1_1-e^{-\ell (I)^2L_1})^{M}
        F\big(\sqrt{L_1},\sqrt{L_2}\big)
        a^{(M, 0)}_R\right| \,d\mu(x)\nonumber\\
    &&\hskip.5cm\leq C\sum_{\kappa=1}^{(2M-1)M}
        \left\|e^{-\kappa\ell (I)^2L_1} (1\!\!1_1-e^{-\ell (I)^2L_1})^{M}
        F\big(\sqrt{L_1},\sqrt{L_2}\big)
        a^{(M, 0)}_R\right\|_{L^1((100 I^* )^c\times 100J)}  \nonumber\\
    &&\hskip.5cm \leq C \left({\ell(I)\over \ell(I^{\ast})}\right)^{s - \frac{n_1}{2}} V(R)^{1/2} \
        \big\|a^{(M, 0)}_R\big\|_{L^2(X_1\times X_2)}.
\end{eqnarray*}
With formula~\eqref{e5.12} and the estimates in {\bf Cases (1)}
and {\bf (2)} in hand, we use the properties of $(H^1_{L_1,
L_2}, 2, N)$-atoms and Journ\'e's covering lemma to get
\begin{eqnarray}\label{ddd}
    D_1
    &\leq&  C  \sum_{m=0}^M\sum_{R\in m(\Omega)}
        \left({\ell(I)\over \ell(I^{\ast})}\right)^{s - \frac{n_1}{2}} V(R)^{1/2}
        \ell(I)^{-2m} \
        \big\|a^{(m, 0)}_R\big\|_2 \nonumber\\
    &\leq& C\sum_{m=0}^M \left(\sum_{R\in m(\Omega)}
        \left({\ell(I)\over \ell(I^{\ast})}\right)^{2s - n_1}V(R) \right)^{1/2}
        \left(\sum_{R\in m(\Omega)} \ell(I)^{-4m} \
        \big\|a^{(m, 0)}_R\big\|_2^2\right)^{1/2}  \nonumber\\
    &\leq & C V(\Omega)^{-1/2} V(\Omega)^{1/2}
    \leq   C
\end{eqnarray}
as desired.

\medskip

\noindent {\bf {\underline{Estimate of $D_2$}}.}\  Recall that
$a_R^{(m_1, m_2)}(x) = (L_1^{M-m_1} \otimes L_2^{M-m_2})
b_R(x)$ for each $m_1$, $m_2 \in \{0, 1, \ldots, M\}$. We use a
similar argument to that in \eqref{e5.11} and~\eqref{e5.12} to
obtain
\begin{align*}
    D_2
    &\leq \!\!\sum_{R\in m(\Omega)}
    \left\{\sum_{m_1=0}^{M-1}\sum_{m_2=0}^{M-1}\ell (I)^{-2(m_1+1)}\ell (J)^{-2(m_2+1)}\!\!\int_{\ell (I)}^{\sqrt{2}\ell (I)}
    \!\!\! \int_{\ell (J)}^{\sqrt{2}\ell (J)} \!\!\!\!w_1 w_2
    \left\|H_{m_1, m_2, w_1,w_2, M}\right\|_{L^1((100 I^* )^c\times (100J)^{c})}
      \,dw_1dw_2 \right.\\
     &\hspace{1.5cm} + \sum_{m_1=0}^{M-1}   \sum_{\kappa_2=1}^{(2M-1)M}\ell (I)^{-2(m_1+1)} \ell (J)^{-2M} \int_{\ell (I)}^{\sqrt{2}\ell (I)}
     w_1  \left\|U_{m_1, w_1,\kappa_2, M}\right\|_{L^1((100 I^* )^c\times (100J)^{c})} \, dw_1\\
     &\hspace{1.5cm} + \sum_{\kappa_1=1}^{(2M-1)M} \sum_{m_2=0}^{M-1}\ell (I)^{-2M}\ell (J)^{-2(m_2+1)}   \int_{\ell (J)}^{\sqrt{2}\ell (J)}
     w_2   \left\|V_{m_2, \kappa_1, w_2, M}\right\|_{L^1((100 I^* )^c\times (100J)^{c})}  \, dw_2\\
     &\hspace{1.5cm} +\left.  \sum_{\kappa_1=1}^{(2M-1)M} \sum_{\kappa_2=1}^{(2M-1)M}
     \ell (I)^{-2M}\ell (J)^{-2M}\left\|Z_{ \kappa_1, \kappa_2, M}\right\|_{L^1((100 I^* )^c\times (100J)^{c})} \right\} \\
     &=: D_{21}+D_{22}+D_{23}+D_{24},
\end{align*}
where
\begin{eqnarray*}
    H_{m_1, m_2, w_1, w_2, M}(x)
    &:=& G_{m_1,   M}(L_1)   (1\!\!1_1-e^{-w_1^2L_1})^M (1\!\!1_2-e^{-w_2^2L_2})^M
        F\big(\sqrt{L_1},\sqrt{L_2}\big)a^{(m_1, m_2)}_R(x),\nonumber\\ [4pt]
    U_{m_1, w_1,\kappa_2, M}(x) &:=&G_{m_1,   M}(L_1)e^{-\kappa_2\ell (J)^2L_2}
        (1\!\!1_1-e^{-w_1^2L_1})^M (1\!\!1_2-e^{-\ell (J)^2L_2})^{M}
        F\big(\sqrt{L_1},\sqrt{L_2}\big)a^{(m_1, M)}_R(x), \nonumber\\ [4pt]
    V_{m_2, \kappa_1, w_2, M}(x) &:=&G_{m_2,   M}(L_2)e^{-\kappa_1\ell (I)^2L_1}
        (1\!\!1_1-e^{-\ell(I)^2L_1})^M (1\!\!1_2-e^{-w_2^2L_2})^{M}
        F\big(\sqrt{L_1},\sqrt{L_2}\big)a^{(M, m_2)}_R(x),\nonumber\\ [4pt]
    Z_{ \kappa_1, \kappa_2, M}(x) &:=& e^{-\kappa_1\ell (I)^2L_1} e^{-\kappa_2\ell (J)^2L_2}
        (1\!\!1_1-e^{-\ell(I)^2L_1})^M (1\!\!1_2-e^{-\ell(J)^2L_2})^M
        F\big(\sqrt{L_1},\sqrt{L_2}\big)a^{(M, M)}_R(x).
\end{eqnarray*}
By (i) of Proposition~\ref{prop4.1}, a similar argument to that
for~$D_1$ shows
\begin{eqnarray*}
    D_{21}
    \leq  C\sum_{R\in m(\Omega)}\sum_{m_1=0}^{M-1}\sum_{m_2=0}^{M-1}\ell (I)^{-2m_1}\ell (J)^{-2m_2}
    \left({\ell(I)\over \ell(I^{\ast})}\right)^{s - \frac{n_1}{2}} V(R)^{1/2} \
    \hspace{1.5cm} \big\|a^{(m_1, m_2)}_R\big\|_{L^2(X_1\times X_2)}.
\end{eqnarray*}
A similar argument to the one above shows that
\begin{eqnarray*}
    && D_{21} + D_{23} + D_{24}\\
    &&\hskip.5cm  \leq C\sum_{R\in m(\Omega)}\left({\ell(I)\over \ell(I^{\ast})}\right)^{s - \frac{n_1}{2}} V(R)^{1/2}
        \left\{  \sum_{m_1=0}^{M-1} \ell (I)^{-2m_1}\ell (J)^{-2M}
        \big\|a^{(m_1, M)}_R\big\|_{L^2(X_1\times X_2)} \right.\\
    &&\hskip1cm  \left. \hspace{2cm} +\sum_{m_2=0}^{M-1} \ell (I)^{-2M}\ell (J)^{-2m_2}
        \big\|a^{(M, m_2)}_R\big\|_{L^2(X_1\times X_2)} +\ell (I)^{-2M}\ell (J)^{-2M}
        \big\|a^{(M, M)}_R\big\|_{L^2(X_1\times X_2)}  \right\}.
\end{eqnarray*}
From the estimates of $D_{21}$, $D_{22}$, $D_{23}$ and
$D_{24}$, H\"older's inequality, Journ\'e's covering lemma and
the properties of $(H^1_{L_1, L_2}, 2, N)$-atoms, we have
\begin{eqnarray}\label{e5.17}
    D_2
    \leq  C V(\Omega)^{-1/2} V(\Omega)^{1/2}
    \leq  C.
\end{eqnarray}
This, together with estimate~\eqref{ddd}, shows that $D\leq C$.
By a symmetric argument, we find that $E\leq C$, which
completes the proof of estimate~\eqref{e5.6}. This proves (i)
of Theorem~\ref{th3.1}.

We then apply Theorem~\ref{th2.7} to interpolate with the $L^2$
result to obtain that the operator $F(L_1, L_2)$ is bounded
from $H_{L_1, L_2}^p( {X_1\times X_2} )$ to $L^p( {X_1\times
X_2} )$ for $1<p\leq 2$. From Proposition~\ref{prop2.3} and by
duality, (ii) of Theorem~\ref{th3.1} follows readily. The proof
of Theorem~\ref{th3.1} is complete.
\end{proof}

\begin{proof}[Proof of Theorem~\ref{th3.2}]\
Using Proposition~\ref{prop4.2} in place of
Proposition~\ref{prop4.1}, the proof of Theorem~\ref{th3.2} is
similar to that of Theorem~\ref{th3.1} with minor
modifications. We omit the details.
\end{proof}

\begin{remark}\label{rem:removecondition}
We note that the one-variable Sobolev conditions~\eqref{e3.3}
and \eqref{e3.4} in Theorem~\ref{th3.1} and \eqref{e3.6} and
\eqref{e3.7} in Theorem~\ref{th3.2} are only used at one point
in the proofs (see details below), and indeed for certain
operators one or both of these one-variable Sobolev assumptions
can be dropped. Specifically, let $L_1$ and $L_2$ be operators
as in Theorem~\ref{th3.1}, and in addition suppose that either
0 is not in the spectrum of $L_1$ or 0 is in the absolutely
continuous spectrum of $L_1$. If $F$ is a bounded Borel
function satisfying the Sobolev conditions \eqref{e3.5}
and~\eqref{e3.3} of Theorem~\ref{th3.1}, then the conclusions
of Theorem~\ref{th3.1} hold; the second one-variable Sobolev
condition~\eqref{e3.4} is not necessary for such an~$F$.
Similarly, for the conclusions of Theorem~\ref{th3.2} to hold,
it suffices for $F$ to satisfy conditions \eqref{e3.8}
and~\eqref{e3.6}; condition~\eqref{e3.7} is not necessary.
Further, by symmetry, the corresponding statements apply if
either 0 is not in the spectrum of $L_2$ or 0 is in the
absolutely continuous spectrum of $L_2$.

We sketch the necessary modifications to the proofs, for $L_1$
and Theorem~\ref{th3.1}; the other cases are similar. If the
spectrum of $L_1$ does not include the point~$0$, or if 0
belongs to the absolutely continuous spectrum of $L_1$, then
the $L^2\to L^2$ operator bound of a spectral multiplier
operator can be controlled by the $L^\infty$ norm of the
multiplier function on the open interval $(0,\infty)$, not
including the point~$0$. Thus in the term $\|\delta_{(1,R_2)}
F\|_{L^{\infty}_{\la_1}({\mathbb R}; \, L^2_{\la_2}({\mathbb
R}))}$ in Lemma~\ref{le4.1}(iii), the $L_{\lambda_1}^\infty$
norm is taken on the open interval $(0,\infty)$, not including
the point~$0$. Therefore in estimate~\eqref{e4.4} in
Proposition~\ref{prop4.1}, the $L_{\lambda_1}^\infty$ norm of
the function $\delta_{(1,R_2)}F$ is taken on the open
interval~$(0,\infty)$, not including the point $0$. So in
estimate~\eqref{eq:one-variablecondition} in the proof of
Theorem~\ref{th3.1}, we can control $\big\|\phi\delta_{(2^\ell,
1)}F_{w,m, M}^{\ell} \big\|_{{L^{\infty}_{\la_2}({\mathbb R};
\, W^{s, 2}_{\la_1}(\mathbb R))}}$ simply by
$\sup_{t_2>0}\big\|\eta_{(1,2)}\delta_{(2^\ell,t_2)} F
\big\|_{W^{(s,1/2+\varepsilon), 2}({\mathbb R}\times {\mathbb
R} )} \cdot\min\left\{1, (2^{\ell}w)^{2M}\right\}$. We no
longer need the term $\big\|\eta_{1}\delta_{(2^\ell,
1)}F(\cdot,0)\big\|_{W^{s,2}({\mathbb R})}$.
\end{remark}

\medskip

\section{Applications}\label{sec:applications}
\setcounter{equation}{0}

As an illustration of our results, we discuss two applications.
Our main results, Theorems~\ref{th3.1} and \ref{th3.2} and
Corollary~\ref{coro3.3}, can be applied to second-order
operators (see for instance \cite{COSY, DOS}), including
standard Laplace operators, Schr\"odinger operators with
inverse-square potential, and sub-Laplacians on homogeneous
groups. Here we present applications to Riesz-transform-like
operators and to double Bochner--Riesz means.

%

\subsection{Riesz-transform-like
operators}\label{sec:Rieszlike} One motivation for
multivariable spectral multipliers comes from operators that
resemble the Riesz transform. For example, the operators
\begin{eqnarray}\label{e6.1}
  {L^{\alpha_1}_1L^{\alpha_2}_2\over (L_1 + L_2)^{\alpha}},
  \ \ \ \ \ \ \alpha = \alpha_1 + \alpha_2 \ \ {\rm with}\
  \alpha_1, \alpha_2 > 0
\end{eqnarray}
and
\begin{eqnarray}\label{e6.2}
   {\sqrt{L_1}\over \sqrt{L_1} + i L_2}
\end{eqnarray}
are known to be bounded on $L^2(X_1\times X_2)$.

Applying Theorem~\ref{th3.1}, we have the following result.

\begin{prop}\label{th6.1}
    Suppose that the metric measure spaces $X_i$, $i = 1$, 2,
    satisfy the doubling condition~\eqref{doubling} with
    exponent~$n_i$. Suppose that  $L_i$, $i = 1$, 2, satisfy the
    finite propagation speed property ${\rm (FS)}$ and restriction
    type estimates $({\rm{{ST}}}_{p_i,2}^{2})$ for some $p_i$ with
    $1\leq p_i<2$. Then
    \begin{itemize}
        \item[(i)] the operators in \eqref{e6.1} and
            \eqref{e6.2} extend to bounded operators from
            $H_{L_1, L_2}^1( {X_1\times X_2} )$ to $L^1(
            {X_1\times X_2} )$, and

        \item[(ii)] the operators in \eqref{e6.1} and
            \eqref{e6.2} are bounded on $L^p( {X_1\times X_2}
            )$ for all $p$ with $\max\{p_1, p_2\} < p \leq 2$.
    \end{itemize}
\end{prop}

\begin{proof}
It can be verified that the multipliers
$$
  {\lambda_1^{2\alpha_1}\lambda_2^{2\alpha_2} \over (\lambda_1^2 + \lambda_2^2)^{\alpha}},
  \ \ \ \ \ \ \alpha=\alpha_1+\alpha_2 \ \ {\rm with}\
  \alpha_1, \alpha_2>0
$$
and
$$
   { \lambda_1 \over  \lambda_1 +i \lambda_2^2}
$$
satisfy conditions \eqref{e3.5}, \eqref{e3.3} and~\eqref{e3.4}
(see paragraph~6.2.4 on page 110, \cite{St1}). The desired
results then follow readily.
\end{proof}

%

\subsection{Double Bochner--Riesz means}\label{sec:doubleBochnerRiesz}
Suppose that the metric measure spaces $X_i$, $i = 1$, 2,
satisfy the doubling condition~\eqref{doubling} with
exponent~$n_i$. Further assume that $L_i$, $i = 1$, 2, are
nonnegative self-adjoint operators satisfying the finite
propagation speed property~(FS).
Define
\begin{equation}\label{e6.3}
    S^{\delta}_{R_1, R_2}(\lambda_1, \lambda_2)
    :=
       \left\{
       \begin{array}{cl}
       \left(1 - \dfrac{\lambda_1^2}{R_1^2} - \dfrac{\lambda_2^2}{R_2^2}\right)^{\delta}
       &\mbox{for}\;\; \dfrac{\lambda_1^2}{R_1^2} + \dfrac{\lambda_2^2}{R_2^2} \le 1, \\ [14pt]
       0  &\mbox{for}\;\; \dfrac{\lambda_1^2}{R_1^2} + \dfrac{\lambda_2^2}{R_2^2} > 1. \\
       \end{array}
      \right.
   \end{equation}
We then define the operator $S^{\delta}_{R_1,
R_2}\big(\sqrt{L_1}, \sqrt{L_2}\big)$ using~\eqref{e1.1}. We
call $S^{\delta}_{R_1, R_2}\big(\sqrt{L_1}, \sqrt{L_2}\big)$
the \emph{double Bochner--Riesz mean of order~$\delta$}. The
basic question in the theory of Bochner--Riesz means is to
establish the critical exponent for uniform continuity with
respect to $R_1, R_2$ and convergence of the Riesz means on
$L^p( {X_1\times X_2} )$ spaces for various~$p$ with $1\le p
\le \infty$.

For ${\delta} = 0$, $S^{0}_{R_1, R_2}\big(\sqrt{L_1},
\sqrt{L_2}\big)$ is the spectral projector $E_{\sqrt{L_1},
\sqrt{L_2}}([0, R_1]\times [0, R_2])$, while for ${\delta}>0$,
$S^{\delta}_{R_1, R_2}\big(\sqrt{L_1}, \sqrt{L_2}\big)$ can be
seen as a smoothed version of this spectral projector.
Bochner--Riesz summability describes the range of ${\delta}$
for which the above operators are bounded on $L^p( {X_1\times
X_2} )$,  uniformly in $R_1, R_2$. Note that if $0<s<{\delta}+
1/2$, then $ S^{\delta}_{R_1, R_2}(\lambda_1, \lambda_2) \in
W^{s,2}(\mathbb{R}^2)$; see~\cite[Lemma~4.4]{BGSY}. Also,
$W^{s,2}(\RR^2) \hookrightarrow W^{(s_1,s_2),2}(\RR\times\RR)$
for $s_1 + s_2 = s$. As a consequence of Theorem~\ref{th3.1}
and Proposition~\ref{prop2.3}, we obtain the following result.

\begin{prop}\label{th6.2}
    Suppose that the metric measure spaces $X_i$, $i = 1$, 2,
    satisfy the doubling condition~\eqref{doubling} with
    exponent~$n_i$. Further assume that $L_i$, $i = 1$, 2, are
    nonnegative self-adjoint operators satisfying the finite
    propagation speed property~${\rm (FS)}$ and restriction
    type estimates~$({\rm{{ST}}}_{p_i,2}^{2})$ for some $p_i$
    with $1\leq p_i<2$. Then for all $\delta>(n_1+n_2)/2 -1/2$,
    \begin{itemize}
        \item[(i)] the operator $S^{\delta}_{R_1,
            R_2}\big(\sqrt{L_1}, \sqrt{L_2}\big)$  extends  to
            a bounded operator  from $H_{L_1, L_2}^1(
            {X_1\times X_2} )$ to $L^1( {X_1\times X_2} )$, and

        \item[(ii)]    $S^{\delta}_{R_1,
            R_2}\big(\sqrt{L_1}, \sqrt{L_2}\big)$ is
            bounded on $L^p( {X_1\times X_2} ) $  for all
            $p_i$ with $\max\{p_1, p_2\}<p\leq 2$. In
            addition,
            \begin{eqnarray*}
                \sup_{R_1, R_2>0} \big\|S^{\delta}_{R_1, R_2}
                    \big(\sqrt{L_1}, \sqrt{L_2}\big)f
                    \big\|_{L^p( {X_1\times X_2} )}
                \leq C\|f\|_{L^p( {X_1\times X_2} )}.
            \end{eqnarray*}
    \end{itemize}
\end{prop}

\bigskip

{\bf Acknowledgments.} P.~Chen, X.T.~Duong, J.~Li and L.A.~Ward
are supported by the Australian Research Council (ARC) under
Grant No.~ARC-DP120100399. J.~Li is also supported by the NNSF
of China, Grant No.~11001275. L.X.~Yan is supported by the NNSF
of China, Grant Nos.~10925106 and~11371378. Part of this work
was done during L.X.~Yan's stay at Macquarie University and
visit to the University of South Australia. L.X.~Yan would like
to thank Macquarie University and the University of South
Australia for their hospitality.


\bigskip


\begin{thebibliography}{99}

\bibitem{Au} P. Auscher, On necessary and sufficient conditions
    for $L^p$-estimates of Riesz transforms associated to
    elliptic operators on $\RR$ and related estimates, {\it
    Memoirs of the Amer. Math. Soc.}~\textbf{186}, no.~871
    (2007), xviii+75 pp.

\bibitem{AMR} P. Auscher, A. M$^{\rm c}$Intosh and E. Russ,
    Hardy spaces of differential forms on Riemannian manifolds,
    {\it J. Geom. Anal.}~{\bf 18} (2008), 192--248.




\bibitem{BGSY} F. Bernicot, L. Grafakos, L. Song and L.X. Yan,
    The bilinear Bochner-Riesz problem. Preprint (2013).

\bibitem{BLYZ} M. Bownik, B.D. Li, D.C. Yang and Y. Zhou, Weighted anisotropic product Hardy spaces and boundedness of sublinear operators, {\it Math. Nachr.}~{\bf 283} (2010), no. 3, 392--442.



\bibitem{CaSe} A. Carbery and A. Seeger, $H^p$- and
    $L^p$-variants of multiparameter Calder\'on-Zygmund theory, {\it Trans.
    Amer. Math. Soc.}~{\bf 334} (1992), 719--747.

\bibitem{C} L. Carleson, A counterexample for measures bounded
    on $H^p$ for the bi-disc, Mittag Leffler Report No.~7,
    1974.

\bibitem{CYZ} D.C. Chang, D.C. Yang and Y. Zhou, Boundedness of sublinear operators on product Hardy spaces and its application, {\it J. Math. Soc. Japan}~{\bf 62} (2010), no. 1, 321--353.

\bibitem{CF1} S.-Y.A. Chang  and R. Fefferman, A continuous
    version of the duality of $H^1$ with {\rm BMO} on the
    bi-disc, {\it Ann. Math.}~{\bf 112} (1980), 179--201.

\bibitem{CF2} S.-Y.A. Chang and R. Fefferman, The
    Calder\'on-Zygmund decomposition on product domains, {\it
    Amer. J. Math.}~{\bf 104} (1982), 445--468.

\bibitem{CGT} J. Cheeger, M. Gromov and M. Taylor, Finite
    propagation speed, kernel estimates for functions of the
    Laplacian and the geometry of complete Riemannian
    manifolds, {\it J. Differential Geom.}~{\bf 17} (1982),
    15--53.

\bibitem{Ch} L.K. Chen, The multiplier operators on the product
    spaces, {\it Illinois J. Math.}~{\bf 38} (1994), 420--433.

\bibitem{CDLWY} P. Chen, X.T. Duong, J. Li, L.A. Ward and L.X.
    Yan, Hardy spaces associated to self-adjoint operators
    satisfying generalized Gaussian estimates on product
    domains. Submitted, 2014.

\bibitem {COSY} P. Chen, E.M. Ouhabaz, A. Sikora and L.X. Yan,
    Endpoint estimates for Bochner-Riesz means and sharp
    spectral multipliers, to appear in Journal d'Analyse Math\'ematique. 

\bibitem{Ch1} M. Christ, A $T(b)$ theorem with remarks on
    analytic capacity and the Cauchy integral, {\it Colloq.\ Math.}~{\bf
    61} (1990), 601--628.

\bibitem {Ch2} M. Christ, $L^p$ bounds for spectral
    multipliers on nilpotent groups, {\it Trans. Amer. Math.
    Soc.}~{\bf 328} (1991), no. 1, 73--81.

\bibitem {CW} R.R. Coifman and G. Weiss, Analyse harmonique
    non-commutative sur certains espaces homog\`enes. \'Etude
    de certaines int\'egrals singuli\`eres, {\it Lecture Notes in
    Math.}~{\bf 242}, Springer-Verlag, Berlin, 1971.

\bibitem{CoSi} T. Coulhon and A. Sikora, Gaussian heat kernel
    upper bounds via Phragm\'en-Lindel\"of theorem, {\it Proc.
    Lond. Math. Soc.}~{\bf 96} (2008), 507--544.

\bibitem{DSTY} D.G. Deng, L. Song, C.Q. Tan and L.X. Yan,
    Duality of Hardy and BMO spaces associated with
    operators with heat kernel bounds on product domains,
    {\it J. Geom. Anal.}~{\bf 17} (2007), 455-483.

\bibitem{DLY}  X.T. Duong, J. Li and L.X. Yan, Endpoint
    estimates for singular integrals with non-smooth kernels on
    product spaces. Submitted, 2014.

\bibitem{DOS}  X.T. Duong, E.M. Ouhabaz and A. Sikora,
    Plancherel-type estimates and sharp spectral multipliers.
    {\it J. Funct. Anal.}~{\bf 196} (2002), 443--485.

\bibitem{DY2} X.T. Duong and L.X. Yan, Duality of Hardy and
    BMO spaces associated with operators with heat kernel
    bounds. {\it J. Amer. Math. Soc.}~{\bf 18} (2005),
    943--973.

\bibitem{DY} X.T. Duong and L.X. Yan, Spectral multipliers
    for Hardy spaces associated to non-negative self-adjoint
    operators satisfying Davies-Gaffney estimates, {\it J.
    Math. Soc. Japan}~{\bf 63} (2011), 295--319.

\bibitem{F} R. Fefferman, Calder\'on-Zygmund theory for
    product domains: $H^p$ spaces, {\it Proc. Nat. Acad. Sci.
    USA}~{\bf 83} (1986), 840--843.

\bibitem{FS} G. Folland and E.M. Stein, {\it Hardy spaces on
    Homogeneous Groups}, Princeton Univ.\ Press, 1982.

\bibitem{GH}  L. Grafakos, P. Honz'k  and A. Seeger, On maximal functions for Mikhlin-Hšrmander multipliers, {\it Adv. Math.}~{\bf 204} (2006), no. 2, 363--378. 


\bibitem{GS} L. Grafakos and Z.Y. Si, The H\"ormander multiplier theorem for multilinear operators, {\it J. Reine Angew. Math.}~{\bf668} (2012), 133--147. 

\bibitem{GHS} C. Guillarmou, A. Hassell and A. Sikora,
    Restriction and spectral multiplier theorems on
    asymptotically conic manifolds, {\it Anal. PDE}~{\bf 6}
    (2013), 893--950.

\bibitem{GS} R. Gundy and E.M. Stein, $H^p$ theory for the
    poly-disc, {\it Proc. Nat. Acad. Sci.}~{\bf 76} (1979),
    1026--1029.
    
    

\bibitem{HLL} Y.S. Han, J. Li and C.C. Lin, Criterions of the $L^2$ boundedness and sharp endpoint estimates for singular integral operators on product spaces of homogeneous type, accepted by Ann. Scuola Norm. Sup. Pisa, (2015), DOI: $10.2422/2036-2145.201411\_002$.

\bibitem{HLMMY} S. Hofmann, G.Z. Lu, D. Mitrea, M. Mitrea and
    L.X. Yan, Hardy spaces associated to non-negative
    self-adjoint operators satisfying Davies-Gaffney
    estimates, {\it Memoirs of Amer. Math. Soc.}~{\bf 214}
    (2011), no.~1007, vi+78 pp.

\bibitem{HM} S. Hofmann and S. Mayboroda, Hardy and BMO spaces
    associated to divergence form elliptic operators, {\it
    Math. Ann.}~{\bf 344} (2009), 37--116.

\bibitem{HMMc} S. Hofmann, S. Mayboroda and A. McIntosh, Second
    order elliptic operators with complex bounded measurable
    coefficients in $L^p$, Sobolev and Hardy spaces, {\it
    Ann. Sci. \'Ecole Norm. Sup.}~{\bf 44} (2011), 723--800.

\bibitem {Ho} L. H\"ormander,  Estimates for translation
    invariant operators in $L^p$ spaces, {\it Acta Math.}~{\bf
    104} (1960), 93--140.

\bibitem{J2} J.-L. Journ\'e, Calder\'on-Zygmund operators on
    product spaces, {\it Rev.~Mat.~Iberoamericana}~{\bf 1}
    (1985), 55--91.

\bibitem{J1} J.-L. Journ\'e, A covering lemma for product
    spaces, {\it Proc. Amer. Math. Soc.}~{\bf 96} (1986),
    593--598.

\bibitem{KU}  P.C. Kunstmann and M. Uhl, Spectral multiplier
    theorems of H\"ormander type on Hardy and Lebesgue spaces,
    to appear  in Analysis \& PDE.

\bibitem {LSV} V. Liskevich, Z. Sobol and H. Vogt, On the $L^p$
    theory of $C^0$-semigroups associated with second-order
    elliptic operators {\rm II}, {\it J. Funct. Anal.}~{\bf
    193} (2002), 55--76.

\bibitem{Mar} A. Martini, Analysis of joint spectral
    multipliers on Lie groups of polynomial growth, \emph{Ann.
    Inst. Fourier (Grenoble)}~{\bf 62} (2012), 1215--1263.

\bibitem {MRS1}  D. M\"uller, F. Ricci and E.M. Stein,
    Marcinkiewicz multipliers and multi-parameter structure on
    Heisenberg (-type) groups, I, {\it Invent. Math.}~{\bf 119}
    (1995), 199--233.

\bibitem {MRS2}  D. M\"uller,  F. Ricci and E.M.  Stein,  Marcinkiewicz multipliers
and multi-parameter structure on Heisenberg (-type) groups, II. {\it
Math. Z.} {\bf 221}  (1996), 267--291.

%
%
%

\bibitem {Ou} E.M. Ouhabaz, {\it Analysis of heat equations on
    domains}, London Math. Soc. Monographs, Vol.~{\bf 31},
    Princeton Univ. Press, 2004.

\bibitem{P} J. Pipher, Journ\'e's covering lemma and its
    extension to higher dimensions, {\it Duke Math. J.}~{\bf
    53} (1986), 683--690.

\bibitem{Si2} A. Sikora, Riesz transform, Gaussian bounds and
    the method of wave equation, {\it Math. Z.}~{\bf 247}
    (2004), 643--662.

\bibitem {Si1} A. Sikora, Multivariable spectral multipliers
    and analysis of quasielliptic operators on fractals, {\it
    Indiana Univ. Math. J.}~{\bf 58} (2009), 317--334.

\bibitem {St1} E.M. Stein, {\it Singular integrals and
    differentiability properties of functions}, Princeton
    Mathematical Series, No.~{\bf 30}, Princeton Univ. Press,
    Princeton, NJ, 1970.

\bibitem {St2} E.M. Stein, {\it Harmonic analysis: Real
    variable methods, orthogonality and oscillatory integrals},
    with the assistance of Timothy S. Murphy, Princeton
    Mathematical Series, No.~{\bf 43}, Monographs in Harmonic
    Analysis, III, Princeton Univ. Press, Princeton, NJ, 1993.


\bibitem {VSC} N. Varopoulos, L. Saloff-Coste and T. Coulhon,
    {\it Analysis and geometry on groups}, Cambridge Univ.
    Press, London, 1993.

\bibitem {Yo} K. Yosida, {\it Functional Analysis} (fifth
    edition), Springer-Verlag, Berlin, 1978.

\end{thebibliography}
\end{document}